\documentclass[a4paper,reqno]{amsart}
\usepackage{amssymb,amscd,amsmath}

\usepackage{upref}
\usepackage[american]{babel}
\usepackage[all]{xy}
\usepackage{fancyhdr}
\usepackage{enumerate}
\usepackage{hyperref}
\usepackage[pageref]{backref}

\theoremstyle{plain}
\newtheorem{theo}{Theorem}[section]

\newtheorem{prop}[theo]{Proposition}
\newtheorem{lem}[theo]{Lemma}
\newtheorem{coro}[theo]{Corollary}

 \theoremstyle{definition}

 \newtheorem{ex}[theo]{Example}

 \theoremstyle{remark}

\numberwithin{equation}{section}

\newcommand{\integers}{\mathbb{Z}}

\DeclareMathOperator{\supp}{supp} 
 
\DeclareMathOperator{\gr}{grad}

\DeclareMathOperator{\Rees}{Rees}

\begin{document}

\title[Free group algebras]{Free group algebras in division rings with valuation II}
\author{Javier S\'anchez}
\address{Department of Mathematics - IME, University of S\~ao Paulo,
Rua do Mat\~ao 1010, S\~ao Paulo, SP, 05508-090, Brazil}
\email{jsanchez@ime.usp.br}
\thanks{Supported by FAPESP-Brazil, Proj. Tem\'atico 2015/09162-9, by 
 Grant CNPq 307638/2015-4 and
by MINECO-FEDER (Spain) 
through project numbers MTM2014-53644-P and MTM2017-83487-P}

\date{\today}

\subjclass[2010]{16K40, 16W70, 16W60, 	16W10,   16S10, 16S30, 16U20, 6F15}

\keywords{Division rings, filtered rings, graded rings, valuations, rings with involution, free group algebra, universal enveloping algebra, ordered groups}

\begin{abstract}
We apply the filtered and graded methods developed in earlier works to find (noncommutative)
free group algebras in division rings. 

If $L$ is a Lie algebra, we denote by $U(L)$ its universal enveloping algebra.  
P. M. Cohn constructed a division ring $\mathfrak{D}_L$ that contains $U(L)$. We denote
by $\mathfrak{D}(L)$ the division subring of $\mathfrak{D}_L$ generated by $U(L)$. 

Let $k$ be a field of characteristic zero and 
 $L$ be a nonabelian Lie $k$-algebra.
If either $L$ is residually nilpotent or $U(L)$ is an Ore domain, we show
that $\mathfrak{D}(L)$ contains (noncommutative) free group algebras. 
In those same cases, if $L$ is equipped with an involution, we are able to prove
that  
the free group algebra in $\mathfrak{D}(L)$ can be chosen generated by symmetric elements
in most cases.

Let $G$ be a nonabelian residually torsion-free nilpotent group and $k(G)$
be the division subring of the Malcev-Neumann series ring generated  
by the group algebra $k[G]$.
If $G$ is equipped with an involution, we show that $k(G)$ contains a (noncommutative) free group
algebra generated by symmetric elements.
\end{abstract}

\maketitle

\setcounter{tocdepth}{1}
\tableofcontents


\section{Introduction}

The search for free objects in division rings has been largely motivated by the
following two conjectures that still remain open:

\begin{itemize}
\item[(G)] If $D$ is a noncommutative division ring, then the 
multiplicative group $D\setminus\{0\}$ contains a free group of rank two;
\item[(A)] If $D$ is a division ring which is infinite dimensional over its center $Z$
and it is finitely generated
(as a division algebra  over $Z$), then $D$ contains a free $Z$-algebra of rank two. 
\end{itemize}
Conjecture (G) was stated by A. I. Lichtman in \cite{Lichtmanfreesubgroupsof} and it has been proved
when the center of $D$ is uncountable \cite{chibafreegroupsinsidedivisionrings} and
when $D$ is finite dimensional over its center \cite{Goncalvesfreegroupsinsubnormal} to name
two important instances where it holds true.
Conjecture (A) was formulated independently
by L. Makar-Limanov in \cite{Makaronfreesubobjects} and T. Stafford. Evidence for conjecture (A)
has been provided in many papers, for example \cite{Makar1}, \cite{Makar2}, \cite{Lichtmanfreeuniversalenveloping},
\cite{BellRogalskifreesubalgebrasoreextensions},  \cite{BellGoncalvesOreexentions}.
 In many division rings in which conjecture (A) holds, $D$ in fact contains a
noncommutative free group $Z$-algebra.  For example, this always happens
if the center of $D$ is uncountable~\cite{GoncalvesShirvani} (or \cite{SanchezObtaininggraded}
for a slightly more general result).
Other examples of the 
existence of free group algebras in division rings
can be found in \cite{Cauchoncorps}, \cite{Makar-LimanovOnsubalgebrasofthe}, 
\cite{Lichtmanfreeuniversalenveloping},
\cite{SanchezfreegroupalgebraMNseries}.
Therefore it makes sense to consider the following unifying
conjecture:
\begin{itemize}
\item[(GA)] Let $D$ be a skew field with center $Z$.
If $D$ is  finitely generated as a division ring over $Z$ and $D$ is
infinite dimensional over $Z$, then $D$ contains a noncommutative
free group $Z$-algebra.
\end{itemize}
For more details on these and related conjectures the reader is
referred to \cite{GoncalvesShirvaniSurvey}.

\medskip

After the work in \cite{GoncalvesShirvaniSurvey},  
\cite{FerreiraGoncalvesHeisenberggroup}, \cite{FerreiraGoncalvesSanchezFreegroupssymmetric},
\cite{FerreiraFornaroliGoncalvesFreeinvolution},
\cite{FerreiraGoncalvesSanchez1}, \cite{FerreiraGoncalvesSanchez2} it has become apparent that
an involutional version of conjectures (G) and (A) should be investigated. Part of our work
deals with an involutional version of (GA).
To be more specific, if
$D$ is equipped with an involution,
under the hypothesis of (GA), can we find a free group algebra whose free
set of generators is formed by symmetric elements (i.e. $x^*=x$)?

\medskip

Let $k$ be a field. A \emph{$k$-involution} on a $k$-algebra $R$ is a
$k$-linear map $*\colon R\rightarrow R$, $x\mapsto x^*$, such that
$(ab)^*=b^*a^*$ for all $a,b\in R$. There are two families of $k$-algebras that usually are equipped with an involution.
These are group $k$-algebras and universal enveloping algebras of Lie $k$-algebras. 
Given an involution on a group (see p.\pageref{involutiongroup} for precise a definition),
it induces a $k$-involution on the
group $k$-algebra $k[G]$ (p.\pageref{involutiongroup}).  Furthermore, if $G$ is an orderable group 
(p.\pageref{orderablegroup}), there is a prescribed
construction of a division $k$-algebra, which we call $k(G)$, that contains $k[G]$,  it  
is generated by $k[G]$ and such that any $k$-involution on $k[G]$ can be
extended to $k(G)$ (see Section~\ref{sec:Heisenberggroup} for more details).
Also, given a $k$-involution (see p.\pageref{kinvolutionLie}) of a Lie $k$-algebra $L$,
it induces a $k$-involution on  the universal enveloping algebra $U(L)$ in the natural way
(\cite[Section~2.2.17]{Dixmierenvelopingalgebras}).  There is also a concrete construction
of a division $k$-algebra, which we denote by $\mathfrak{D}(L)$. It contains
$U(L)$, it is generated by $U(L)$ and such that any
$k$-involution on $L$ can be extended to a $k$-involution of $\mathfrak{D}(L)$ (see Section~\ref{sec:residuallynilpotent}
for more details).
 We remark that neither $k[G]$ nor $U(L)$ need to be Ore domains, but
if they are, both $k(G)$ and $\mathfrak{D}(L)$ coincide with the Ore ring of fractions
of $k[G]$ and $U(L)$ respectively.

\medskip

The aim of our work is to apply the graded and filtered methods developed in \cite{SanchezObtaininggraded}
and \cite{SanchezObtainingI} to obtain free group algebras in division rings.
Concerning conjecture (GA), we are able to prove an extension of
a result by Lichtman. 
More precisely, \cite[Theorem~4]{Lichtmanfreeuniversalenveloping} is
(2) of the following result.

\begin{theo}\label{theo:Lichtman}
Let $k$ be a field of characteristic zero and $L$ be a nonabelian Lie $k$-algebra.
If one of the following conditions is satisfied 
\begin{enumerate}[\rm(1)]
\item $L$ is residually nilpotent; 
\item The universal enveloping algebra $U(L)$
is an Ore domain; 
\end{enumerate}
then $\mathfrak{D}(L)$ contains a (noncommutative) free group
$k$-algebra. \qed 
\end{theo}
Notice that $\mathfrak{D}(L)$ may not contain a free $k$-algebra of rank two if the characteristic of $k$
is not zero. In fact, as noted in \cite[p.147]{Lichtmanfreeuniversalenveloping}, the
proof given in \cite[p.204]{JacobsonLiealgebras} shows that if $L$ is finite dimensional over $k$,
then $\mathfrak{D}(L)$ is finite dimensional over its center. Therefore, it does not contain a
noncommutative free algebra.

\medskip

Concerning involutional versions of conjecture (GA), we are able to prove the following two theorems.

\begin{theo}\label{theo:intro2}
Let $k$ be a field of characteristic zero and $L$ be a nonabelian Lie $k$-algebra
endowed with a $k$-involution $*\colon L\rightarrow L$. Suppose that one of the following
conditions is satisfied.
\begin{enumerate}[\rm(1)]
\item $L$ is residually nilpotent;
\item The universal enveloping algebra $U(L)$ is an Ore domain and either
\begin{enumerate}[\rm(a)]
	\item there exists $x\in L$ such that $[x^*,x]\neq 0$ and the Lie $k$-subalgebra
	of $L$ generated by $\{x,x^*\}$ is of dimension at least three, or
	\item  $[x^*,x]=0$ for every $x\in L$, but there exist $x,y\in L$ with $[y,x]\neq 0$
	and the $k$-subspace of $L$ spanned by $\{x,x^*,y,y^*\}$ is not equal to the
	Lie $k$-subalgebra of $L$ generated by $\{x,x^*,y,y^*\}$.
	\end{enumerate}
\end{enumerate}
Then $\mathfrak{D}(L)$ contains a (noncommutative) free group $k$-algebra whose free
generators are symmetric with respect to the extension of $*$ to $\mathfrak{D}(L)$. \qed
\end{theo}

\begin{theo}\label{theo:residuallynilpotentgroup}
Let $k$ be a field of characteristic zero and $G$ be a nonabelian residually torsion-free nilpotent 
group endowed with an involution $*\colon G\rightarrow G$. Then $k(G)$
contains a free group $k$-algebra whose free generators are symmetric with respect to the
extension of $*$ to $k(G)$. \qed
\end{theo}

Notice that, since the map $L\mapsto L$, $x\mapsto -x$, is a $k$-involution
for any Lie $k$-algebra $L$, then Theorem~\ref{theo:intro2}(1) implies Theorem~\ref{theo:Lichtman}(1).
On the other hand, the proofs and the 
elements that generate the free group algebra in Theoren~\ref{theo:intro2}
are more complicated than those of Theorem~\ref{theo:Lichtman}.

Let $k$ be a field of characteristic zero.
The general strategy to obtain Theorems~\ref{theo:Lichtman}
and \ref{theo:intro2} goes back to Lichtman~\cite{Lichtmanfreeuniversalenveloping},
and it was also used in \cite{FerreiraGoncalvesSanchez2}.
Roughly speaking, one has to obtain free (group) algebras in the division ring 
$\mathfrak{D}(H)$, where $H$ is the 
Lie $k$-algebra $H=\langle x,y\colon 
[y,[y,x]]=[x,[y,x]]=0\rangle$. From this, one obtains free group algebras in 
$\mathfrak{D}(L)$ where $L$ is a residually nilpotent Lie $k$-algebra. Now
there is a way  to obtain free (group) algebras
in $\mathfrak{D}(L)$, where $L$ is a Lie $k$-algebra
such that $U(L)$ is an Ore domain, from the residually nilpotent case using filtered and graded methods.
We have improved and somehow clarified this strategy in order to obtain the two first theorems above.
Then 
Theorem~\ref{theo:residuallynilpotentgroup} is obtained from the previous results using the filtered 
methods from \cite{SanchezObtainingI} and a technique from \cite{FerreiraGoncalvesSanchez2}.

\bigskip

We begin Section~\ref{sec:generalfiltrations} introducing some basics on filtrations and valuations. In
Section~\ref{sec:filtrationuniversal}, we state some results on how filtrations and gradations of Lie algebras
induce filtrations and gradations of their universal enveloping algebras. Section~\ref{sec:freegroupalgebrasdivision}
is devoted to results about the existence of free group algebras obtained in \cite{SanchezObtaininggraded} and \cite{SanchezObtainingI}. They
show different ways of obtaining free group algebras in division rings generated by group graded rings,
and in division rings endowed with a valuation. 

The results in Section~\ref{sec:filtrations&valuations} are stated in more generality than necessary in
subsequent sections, but we believe there is some merit in the general statements and they could be of interest to others.

The first part of Section~\ref{sec:nilpotentLieinvolutions} is concerned with the classifications of all the 
$k$-involutions of the Heisenberg Lie $k$-algebra $H=\langle x,y\colon [y,[y,x]]=[x,[y,x]]=0 \rangle$
over $k$, a field of characteristic different from two. We are able to prove that, up to equivalence,
there are three involutions on $H$. We then use this to show that any nilpotent Lie $k$-algebra
endowed with an involution $*\colon L\rightarrow L$ contains a $*$-invariant $k$-subalgebra
$H$ of $L$ and whose restriction to $H$ is one of those three involutions.

Section~\ref{sec:nilpotentOre} deals with the problem of finding free (group) algebras
in the Ore ring of fractions of $U(L)$, the universal enveloping algebra of a nilpotent
Lie $k$-algebra $L$ over a field of characteristic zero. The main result is the technical
Theorem~\ref{theo:freegroupHeisenberg}, where a lot of free (group) algebras in $\mathfrak{D}(H)$
are obtained. Each of those free (group) algebras is suitable for later applications of the results
in Section~\ref{sec:freegroupalgebrasdivision}. Thus the free generators (or elements obtained from them)
will be homogeneous elements of some graded rings that appear in this and subsequent sections.
There could be simpler elements that do the job and avoid some technicalities. But we were not able to find them.

Let $L$ be a nonabelian residually nilpotent Lie $k$-algebra over a field of characteristic zero $k$. The
main aim of Section~\ref{sec:residuallynilpotent} is to obtain free (group) algebras
in the division ring $\mathfrak{D}(L)$ from the ones obtained in the previous section. It is done by a method
involving series that was developed in \cite{FerreiraGoncalvesSanchez2}. Although technical,
the argument is quite natural.

Let $L$ be a nonabelian Lie $k$-algebra over a field of characteristic zero
such that its universal enveloping algebra $U(L)$ is an Ore domain. In Section~\ref{sec:Ore},
we find free group algebras in $\mathfrak{D}(L)$, the Ore ring of fractions of $U(L)$, using the results in
previous sections. Roughly speaking, the idea of the proof is that for some natural filtrations of $L$, the associated
graded Lie algebra $\gr(L)$ is  residually nilpotent.  
The isomorphism of graded algebras 
$U(\gr(L))\cong \gr(U(L))$ allows us to use the results in previous sections
thanks to the fact that $U(L)$ is an Ore domain and the good behaviour of the Ore localization
with respect to filtrations described in Section~\ref{sec:filtrations&valuations}.

The arguments in Section~\ref{sec:Ore} should clarify why some of the elements in earlier sections where chosen
in that way. Here it is one of the places where Proposition~\ref{prop:freeobjecthomogeneous} and
Theorem~\ref{coro:divisionrings} are strongly used.

The last section of the paper is devoted to finding free group algebras in $k(G)$ for $k$ a field
of characteristic zero and $G$ a nonabelian residually torsion-free nilpotent group. 
Let $\mathbb{H}=\langle a,b\colon (b,(b,a))=(a,(b,a))=1\rangle$ be
the Heisenberg group. There are filtrations of the group ring $k[\mathbb{H}]$ such
that the induced $k$-algebra is isomorphic to $U(H)$ as graded $k$-algebras, where
 we consider a certain gradation in $U(H)$ induced from one of $H$. Again using the crucial
results of Section~\ref{sec:freegroupalgebrasdivision}, one can obtain suitable free group
algebras in $k(\mathbb{H})$. From this, using an argument from \cite{FerreiraGoncalvesSanchezFreegroupssymmetric}
one gets the desired free group algebras in $k(G)$.

\section{Filtrations, gradations and valuations}\label{sec:filtrations&valuations}

A \emph{strict ordering} on a set $S$ is a binary relation $<$ which is transitive and such that
$s_1<s_2$ and $s_2<s_1$ cannot both hold for elements $s_1,s_2\in S$. It is a \emph{strict total
ordering} if for every $s_1,s_2\in S$ exactly one of $s_1<s_2$, $s_2<s_1$ or $s_1=s_2$ holds.

A group $G$ is called \emph{orderable} \label{orderablegroup} if its elements can be given a strict total ordering $<$
which is left and right invariant. That is, $g_1<g_2$ implies that $g_1h<g_2h$ and $hg_1<hg_2$ for all $g_1,g_2, h\in G$.
We call the pair $(G,<)$ an \emph{ordered group}.
Clearly, any additive subgroup of the real numbers is orderable. More generally, torsion-free abelian groups,
torsion-free nilpotent groups and residually torsion-free nilpotent groups are orderable \cite{Fuchs}.

We would like to point out  that the results in this section are stated for ordered groups, 
but in the following ones they will be applied for the ordered group $\mathbb{Z}$ alone. 
We believe there is some merit in the general statements and they could be of interest to others.

\subsection{On filtrations and valuations}\label{sec:generalfiltrations}

Let $R$ be a ring and $(G,<)$ an ordered group. A family $F_GR=\{F_gR\}_{g\in G}$ of additive subgroups
of $R$ is a (descending)
\emph{$G$-filtration} if it satisfies the following four conditions
\begin{enumerate}[(F1)] 
\item $F_gR\supseteq F_hR$ for all $g,h\in G$ with $g\leq h$;
\item $F_gR\cdot F_hR\subseteq F_{gh}R$ for all $g,h\in G$;
\item $1\in F_1R$;
\item $\bigcup\limits_{g\in G}F_gR=R$.
\end{enumerate}
We say that the $G$-filtration is \emph{separating} if it also satisfies
\begin{enumerate}[(F5)]
\item For every $x\in R$, there exists $g\in G$ such that $x\in F_gR$ and
$x\notin F_{h}R$ for all $h\in G$ with $g<h$.
\end{enumerate}

\medskip

Let $R$ be a ring, $(G,<)$ be an ordered group and $F_GR=\{F_gR\}_{g\in G}$ be a $G$-filtration of $R$.
For each $g\in G$, define $F_{>g}R=\sum_{h>g}F_hR$ and
$$R_g=F_gR/F_{>g}R.$$  The fact that $G$ is an ordered group and the definition of $G$-filtration imply that 
$$F_{>g}R\cdot F_{>h}R\subseteq
F_{>gh}R,\
 F_{>h}R\cdot F_{\geq g}R\subseteq F_{>gh}R,\ F_{\geq g}R\cdot F_{>h}R\subseteq
F_{>gh}R$$ for any
$g,h\in G$. Thus a multiplication can be defined by
\begin{equation}\label{eq:filtrationmultiplication}
R_g\times R_h\longrightarrow R_{gh},\quad
(x+F_{>g}R)(y+F_{>h}R)=xy+F_{>gh}R.
\end{equation}
The \emph{associated graded ring} of $F_GR$ is defined to
be
$$\gr_{F_G}(R)=\bigoplus_{g\in G} R_g.$$ The addition on $\gr_{F_G}(R)$ arises
from the addition on each component $R_g$. The multiplication is
defined by extending the multiplication
\eqref{eq:filtrationmultiplication} on the components bilinearly to
all $\gr_{F_G}(R)$. Notice that $\gr_{F_G}(R)$ may not have
an identity element. If $F_GR$ is separating, then 
$\gr_{F_G}(R)$ is a ring with identity element $1+F_{>1}R$.

The \emph{Rees ring} of the filtration is 
$$\Rees_{F_G}(R)=\bigoplus_{g\in G}(F_gR) g,$$
which is a subring of the group ring $R[G]$. 
Thus an element of $\Rees_{F_G}(R)$ is a finite sum
$\sum_{g\in G}a_gg$ where $a_g\in F_gR$.
Notice that $\Rees_{F_G}(R)$
is a $G$-graded ring with identity element $1_{\Rees_{F_G}(R)}=1_R1_G$.

The next lemma is well known. It can be proved as in \cite[Section~1.8]{MarubayashiVanOystaeyen} where
the filtrations are ascending.

\begin{lem}\label{lem:Reesgradedring}
Let $R$ be a ring, $(G,<)$ be an ordered group and
$F_GR=\{F_gR\}_{g\in G}$ be a $G$-filtration of $R$. 
The following hold true.
\begin{enumerate}[\rm(1)] 
\item The subset $G_1=\{g\in G\colon g\leq 1\}$ is an Ore subset of $\Rees_v(R)$ and
the Ore localization ${G_1}^{-1}\Rees_{F_G} (R)=R[G]$, the group ring.
\item Let $J$ be the ideal of $\Rees_{F_G}(R)$
generated by $G^-=\{g\in G\colon g<1\}$. 
Then $J=\bigoplus\limits_{g\in G}(F_{>g}R)g$ and $\Rees_{F_G}(R)/J\cong \gr_{F_G}(R)$ as graded rings.
\item Let $I$ be the ideal of $\Rees_{F_G}(R)$ generated by the elements
$\{1-g\colon g\in G^-\}$. Then $\Rees_{F_G}(R)/I\cong R$. \qed
\end{enumerate}
\end{lem}

Let $R$ be a ring and $(G,<)$ be an ordered group. A map $\upsilon\colon R\rightarrow G\cup\{\infty\}$
is a \emph{valuation} if it satisfies
\begin{enumerate}[(V1)]
\item $\upsilon(x)=\infty$ if, and only if, $x=0$;
\item $\upsilon(x+y)\geq \min\{\upsilon(x),\upsilon(y)\}$;
\item $\upsilon(xy)=\upsilon(x)\upsilon(y)$.
\end{enumerate}
Notice that $\upsilon(1)=1_G$ and $\upsilon(-x)=\upsilon(x)$
for all $x\in R$. 
For each $g\in G$, we set  $R_{\geq g}=\{f\in R\colon \upsilon(f)\geq g\}$
and $R_{>g}=\{f\in R\colon \upsilon(f)>g\}$.
Defining $F_gR=R_{\geq g}$ for each $g\in G$,
we obtain a separating filtration $F_GR=\{F_gR\}_{g\in G}$. We will denote
the graded ring and the Rees ring associated to this filtration as
$\gr_\upsilon(R)$ and $\Rees_\upsilon(R)$, respectively. Furthermore,
observe that $\gr_\upsilon(R)$ is a domain because of (V3). 
It is well known that the converse is also true \cite[p.91]{MarubayashiVanOystaeyen}. 
That is, given a 
separating filtration $F_GR=\{F_gR\}_{g\in G}$ of $R$ such that
the associated graded ring $\gr_{F_G}(R)$ is a domain, one can
define a valuation
$\upsilon\colon R\rightarrow G\cup\{\infty\}$ by $\upsilon(x)=\max\{g\in G\colon x\in F_gR\}$
for each $x\in R\setminus\{0\}$.

If $X$ is an Ore domain, by $Q_{cl}(X)$, we denote the Ore ring of fractions of $X$
that is, the Ore localization of $X$ at the multiplicative set $X\setminus\{0\}$.

The following lemma is a generalization of 
\cite[Propositions~16,17,18]{Lichtmanfreeuniversalenveloping}, with a
somewhat different proof.

\begin{lem}\label{lem:gradedOre}
Let $R$ be an Ore domain, $(G,<)$ be an ordered group and
$\upsilon\colon R\rightarrow G\cup\{\infty\}$ be a valuation. 
Let $D$ be the Ore ring of fractions of $R$. The
following hold true.
\begin{enumerate}[\rm(1)]
\item The valuation $\upsilon$ can be extended to a valuation $\upsilon\colon D\rightarrow
G\cup\{\infty\}$.
\item The set $\mathcal{H}$ of nonzero homogeneous elements
of $\gr_\upsilon(R)$ is an Ore subset of $\gr_\upsilon(R)$.
\item There exists an isomorphism of $G$-graded rings
$\lambda\colon \mathcal{H}^{-1}\gr_\upsilon(R)\rightarrow \gr_\upsilon(D)$
given by $f+R_{>\upsilon(f)}\mapsto f+D_{>\upsilon(f)}$ for all
$f\in R$.
\item If $G$ is poly-(torsion-free abelian), then $\gr_\upsilon(R)$ is an Ore domain.
\item If $G$ is poly-(torsion-free abelian), then $\Rees_\upsilon(R)$ is an Ore domain.
\item If $G$ is torsion-free abelian, then $\mathcal{J}=\Rees_\upsilon(R)\setminus J$ is
an Ore subset of $\Rees_\upsilon(R)$ and $\mathcal{J}^{-1}\Rees_\upsilon(R)$ is a local ring with
residue division ring $Q_{cl}(\gr_\upsilon(R))$. 
\end{enumerate}
\end{lem}

\begin{proof}
The proof of (1) can be found in \cite[Proposition~9.1.1]{Cohnskew}
for example.

(2) Let $f_1,f_2\in R\setminus\{0\}$. Consider the nonzero homogeneous elements
 $f_1+R_{>\upsilon(f_1)}, f_2+R_{>\upsilon(f_2)}\in\gr_\upsilon
 (R)$. Since $R$ is an Ore domain, there exist $q_1,q_2\in R$ such
 that $q_1f_1=q_2f_2\neq 0$. Consider the nonzero homogeneous
 elements $q_1+R_{>\upsilon(q_1)}, q_2+R_{>\upsilon(q_2)}\in\gr_\upsilon
 (R)$. Then $$(q_1+R_{>\upsilon(q_1)})(f_1+R_{>\upsilon(f_1)})=
 (q_2+R_{>\upsilon(q_2)})(f_2+R_{>\upsilon(f_2)}).$$
 Now \cite[Lemma~8.1.1]{NastasescuvanOystaeyenMethodsgraded} implies
 the result.

(3) First note that $\gr_\upsilon(D)$ is a $G$-graded skew field,
and the natural maps 
$\iota\colon\gr_\upsilon(R)\hookrightarrow
\gr_\upsilon(D)$, $\kappa\colon \gr_\upsilon(R)\hookrightarrow
\mathcal{H}^{-1}\gr_\upsilon(R)$ are embeddings of $G$-graded rings. Thus, for each element
in $\mathcal{H}$, the image by $\iota$ is an homogeneous invertible
element in $\gr_\upsilon(D)$. By the universal property of the Ore
localization, there exists a homomorphism $\lambda\colon
\mathcal{H}^{-1}\gr_\upsilon(R)\rightarrow \gr_\upsilon(D)$ such
that $\iota=\lambda\kappa$. The homomorphism $\lambda$ is injective
since it is so when restricted to homogeneous elements. Let now
$f,q\in R\setminus\{0\}$. Consider $q^{-1}f+D_{>\upsilon(q^{-1}f)}$. This
element is the image by $\lambda$ of
$(q+R_{>\upsilon(q)})^{-1}(f+R_{>\upsilon(f)})$. Thus $\lambda$ is
surjective.

(4) The graded division ring $\gr_\upsilon(D)$ is a crossed product
of the division ring $D_0$ over the subgroup $\{g\in G\colon D_g\neq 0\}$, 
which is again poly-(torsion-free abelian). Thus $\gr_\upsilon(D)$ is an Ore domain by,
for example, \cite[Corollary~37.11]{Passman1}. We show that the Ore ring of fractions
$Q_{cl}(\gr_\upsilon(D))$ of $\gr_\upsilon(D)$ is also the Ore ring of fractions
of $\gr_\upsilon(R)$. For that, it is enough to show that every element of $Q_{cl}(\gr_\upsilon(D))$
is of the form $b^{-1}a$ with $a,b\in Q_{cl}(\gr_\upsilon(R))$, $b\neq 0$.
An element of $f\in Q_{cl}(\gr_\upsilon(D))$ is of the form
$(d_{g_1}+\dotsb +d_{g_r})^{-1}(e_{h_1}+\dotsb+e_{h_s})$ where
$d_{g_i}\in D_{g_i}$, $e_{h_j}\in D_{h_j}$. By (2),(3) and after bringing
to a common denominator, we may suppose that there exist
$t,a_i,b_j\in \mathcal{H}$ such that 
$$f=(t^{-1}a_1+\dotsb+t^{-1}a_r)^{-1}(t^{-1}b_1+\dotsb+t^{-1}b_s)=
(a_1+\dotsb+a_r)^{-1}(b_1+\dotsb+b_s).$$

(5) In the same way as (4), one can show that the group ring 
$D[G]$ and  $R[G]$ are Ore domains with the same Ore ring of fractions $Q_{cl}(R[G])$.
By Lemma~\ref{lem:Reesgradedring}(1), $R[G]$ is the localization of
$\Rees_\upsilon(R)$ at $G_1$. Hence one can proceed as in (4) to show
that $\Rees_\upsilon(R)$ is an Ore domain with Ore ring of fractions $Q_{cl}(R[G])$.

(6) Let $x=\sum\limits_{i=1}^n a_ig_i\in\Rees_\upsilon(R)$ where
we suppose that $a_i\neq 0$, $i=1,\dotsc,n$. Hence $\upsilon(a_i)\geq g_i$ for all $i$.
We suppose that if $i<j$ either $\upsilon(a_i)^{-1}g_i<\upsilon(a_j)^{-1}g_j$ or
$\upsilon(a_i)^{-1}g_i=\upsilon(a_j)^{-1}g_j$ and $g_i<g_j$. We define
$\omega(x)=\upsilon(a_n)^{-1}g_n\leq 1_G$. Observe that 
$x=x'\omega(x)$, where $$x'=\sum_{i=1}^n a_ig_ig_n^{-1}\upsilon(a_n).$$ Since
$\upsilon(a_n)^{-1}g_n\geq \upsilon(a_i)^{-1}g_i$, then 
$g_i^{-1}\upsilon(a_i)\geq g_n^{-1}\upsilon(a_n)$. It implies that 
$\upsilon(a_i)\geq g_ig_n^{-1}\upsilon(a_n)$.
Hence $x'\in \mathcal{J}$. Note that $x\in\mathcal{J}$ if and only if $\omega(x)=1$, since
$J=\bigoplus_{g\in G}F_{>g}R\cdot g$.

If $a,b\in R$, $g\in G$ such that $ag,bg\in\Rees_\upsilon(R)$,  then
 $$\omega((a+b)g)\leq
\max\{\omega(ag),\omega(bg)\}.$$ Let now
$y=\sum\limits_{j=1}^pb_jh_j\in \Rees_{\upsilon}(R)$
where we suppose that  $b_j\neq 0$, $j=1,\dotsc,p$, and 
if $j<l$ either $\upsilon(b_j)^{-1}h_j<\upsilon(b_l)^{-1}h_l$ or
$\upsilon(b_j)^{-1}h_j=\upsilon(b_l)^{-1}h_l$ and $h_j<h_l$. 

Now note that $$xy=a_nb_pg_nh_p+\sum_{\{(i,j)\colon (i,j)\neq(n,p)\}}a_ib_jg_ih_j,$$
and if $(i,j)\neq(n,p)$, either $\upsilon(a_ib_j)^{-1}g_ih_j<\upsilon(a_nb_p)^{-1}g_nh_p$
or $g_ih_j<g_nh_p$. Therefore $\omega(xy)=\omega(x)\omega(y)$.

Let $u\in\mathcal{J}$ and $v\in\Rees_\upsilon(R)$. 
 Since $\Rees_\upsilon(R)$ is an Ore domain,
there exist $x,y\in \Rees_\upsilon(R)$ such that $xu=yv$. We have to prove that
$y$ can be chosen such that
$y\in\mathcal{J}$. From $xu=x'u\omega(x)=y'v'\omega(y)\omega(v)=yv$,
we get $x\omega(y)^{-1}u=y'v$, where $y'\in\mathcal{J}$. 
Since $\omega(x)=\omega(y)\omega(v)$ with $\omega(v)\leq 1$,
then $\omega(x)\leq \omega(y)\leq 1$. It implies that 
$x\omega(y)^{-1}\in\Rees_\upsilon(R)$ and $\omega(y')=1$.
\end{proof}

\subsection{On gradations and filtrations of universal enveloping algebras}\label{sec:filtrationuniversal}

If $L$ is a Lie algebra, we denote its
\emph{universal enveloping algebra} by $U(L)$.

Let $k$ be a field, $L$ be a Lie $k$-algebra and  $G$ be a commutative group. We
say that $L$ is a \emph{$G$-graded Lie $k$-algebra} if there exists
a decomposition of $L$ as $L=\bigoplus\limits_{g\in G}L_g$ satisfying
\begin{enumerate}
\item $L_g$ is a $k$-subspace of $L$ for each $g\in G$,
\item $[L_g,L_h]\subseteq L_{g+h}$ for all $g,h\in G$.
\end{enumerate}
The elements of $\bigcup\limits_{g\in G} L_g$ are the
\emph{homogeneous elements} of $L$. If $x\in L_g$, we say that $x$
is \emph{homogeneous of degree} $g$.

The main examples we will deal with are the following.
Examples~(a),(b) are important
in Section~\ref{sec:nilpotentOre}, while examples (c),(d) are useful in Section~\ref{sec:Heisenberggroup}

\begin{ex}\label{ex:gradedLie}
Let $k$ be a field. We can endow the Heisenberg Lie $k$-algebra $H$
with different $\mathbb{Z}$-gradings. We will use the following ones.
\begin{enumerate}[(a)]
\item $H=\bigoplus_{n\in\mathbb{Z}}H_n$, where $H_{-1}=kx+ky$, $H_{-2}=kz$ and $H_n=0$ for all $n\neq -1,-2.$
\item $H=\bigoplus_{n\in\mathbb{Z}}H_n$, where $H_{-1}=kx,$ $H_{-2}=ky$, $H_{-3}=kz$ and $H_n=0$ for all $n\neq -1,-2,-3$.
\item $H=\bigoplus_{n\in\mathbb{Z}}H_n$, where $H_{1}=kx+ky$, $H_{2}=kz$ and $H_n=0$ for all $n\neq 1,2.$
\item $H=\bigoplus_{n\in\mathbb{Z}}H_n$, where $H_{1}=kx,$ $H_{2}=ky$, $H_{3}=kz$ and $H_n=0$ for all $n\neq 1,2,3$.
\qed
\end{enumerate}
\end{ex}

For each $g\in G$, let $\mathcal{B}_g=\{e^g_i\}_{i\in I_g}$ be a
$k$-basis of $L_g$. Then $\mathcal{B}=\bigcup_{g\in G}\mathcal{B}_g$
is a $k$-basis of $L$. Fix an ordering $<$ of $\mathcal{B}$.
Consider the universal enveloping algebra $U(L)$ of $L$. The
\emph{standard monomials} in $\mathcal{B}$ are the elements
\begin{equation}\label{eq:standardmonomials} 
e^{g_1}_{i_1}e^{g_2}_{i_2}\dotsb e^{g_r}_{i_r}\in U(L), \textrm{
with } e^{g_j}_{i_j}\in\mathcal{B}_{g_j},\ e^{g_1}_{i_1}\leq
e^{g_2}_{i_2}\leq \dotsb \leq e^{g_r}_{i_r}.
\end{equation}
By the Poincar\'e-Birkoff-Witt (PBW) theorem, the standard monomials, together with 1, form a
$k$-basis of $U(L)$. We say that the standard monomial
\eqref{eq:standardmonomials} is of degree $g=g_1+g_2+\dotsb+g_r.$
In this situation, one can obtain a gradation of the universal enveloping algebra as follows.

\begin{lem}\label{lem:gradeduniversalenveloping}
Let $G$ be a group and $L=\bigoplus\limits_{g\in G}L_g$ be a
$G$-graded Lie $k$-algebra. Then the universal enveloping algebra
$U(L)$ is an (associative) $G$-graded $k$-algebra. Indeed,
$$U(L)=\bigoplus_{g\in G} U(L)_g,$$ where 
 $U(L)_g=k\textrm{-span of
the standard monomials of degree }g$. \qed
\end{lem}

\bigskip

Let $k$ be a field,  $L$ be a Lie $k$-algebra and $(G,<)$ be an ordered abelian group.
A (descending)  \emph{separating filtration} of $L$ is a family of subspaces $F_GL=\{F_gL\}_{g\in G}$, such that
\begin{enumerate}[(FL1)]
    \item $F_gL\supseteq F_hL$ for all $g,h\in G$ with $g\leq h$;
\item $[F_gL, F_hL]\subseteq F_{g+h}R$ for all $g,h\in G$;
\item $\bigcup_{g\in G}F_gL=L$;
\item For every $x\in L$, there exists $g\in G$ such that $x\in F_gL$ and
$x\notin F_{h}L$ for all $h\in G$ with $g<h$.
\end{enumerate}

Define $F_{>g}L=\sum_{h>g}F_hL$, and $L_g=F_gL/F_{>g}L$ for all $g\in G$. Then
one obtains the \emph{associated graded Lie $k$-algebra}
$$\gr_{F_G}{L}=\bigoplus_{g\in G} L_g.$$

The filtration $F_GL$ of $L$ induces a filtration $F_GU(L)=\{F_gU(L)\}_{g\in G}$ 
of the universal enveloping algebra
$U(L)$ as follows. Define, for each $g\in G$, $g\leq 0$,
$$F_gU(L)=k+\sum\limits_{g_1+\dotsb+g_r\geq g} L_{g_1}\dotsb L_{g_r},$$
and for each $g>0$
$$F_gU(L)=\sum_{g_1+\dotsb+g_r\geq g} L_{g_1}\dotsb L_{g_r}.$$
Then, $F_hU(L)\subseteq F_gU(L)$ for $g<h$, and $F_gU(L)
\cdot F_hU(L)\subseteq F_{g+h}U(L)$ for all $g,h\in G$.

An easy but important example for us is the following. It will be used in Section~\ref{sec:Ore}.
\begin{ex}\label{ex:usualfiltrationLiealgebra}
Let $L$ be a Lie $k$-algebra generated by two elements $u,v\in L$. Define $FL_r=0$ for all $r\geq 0$,  
$FL_{-1}=ku+kv$, and, for $n\leq -1$,
$$F_{n-1}L=\sum_{n_1+n_2+\dotsb+n_r\geq (n-1)}[F_{n_1}L,[F_{n_2}L,\dotsc] \dotsb].$$
Observe that, for each $n\in\mathbb{Z}$, there exists   $\mathcal{B}_{n}\subseteq L$
whose classes give a basis of $L_{n}=F_nL/F_{n+1}L$ such that
$\bigcup\limits_{n\in\integers} \mathcal{B}_{n}$ is a basis of $L$.  \qed
\end{ex}

The next lemma will be used in Sections~\ref{sec:Ore},\ref{sec:Heisenberggroup}

\begin{lem}\label{lem:filtrationuniversalenveloping}
Let $k$ be a field and $L$ be a Lie $k$-algebra.
The following hold true.
\begin{enumerate}[\rm(1)] 
\item Suppose that there exists a basis $\mathcal{B}_g=\{e^g_i\}_{i\in
I_g}$ of $L_g$ for each $g\in G$ such that $\bigcup\limits_{g\in
G}\mathcal{B}_g$ is a basis of $L$. Then the filtration is separating and
there exists an isomorphism of $G$-graded $k$-algebras
$$U(\gr_{F_G}(L))\cong\gr_{F_G}(U(L)).$$
Hence the filtration induces a valuation $\upsilon \colon U(L)\rightarrow G\cup\{\infty\}$.
\item If $U(L)$ is an Ore domain, then $U(\gr_{F_G}(L))$ is an
Ore domain.
\end{enumerate}
\end{lem}
\begin{proof}
(1) It can be proved in the same way as \cite[Proposition~1]{Vergne} or
\cite[Lemma~2.1.2]{Boiscorpsenveloppants}.

(2) By Lemma~\ref{lem:gradedOre}(4).
\end{proof}

\subsection{Free group algebras in division rings}\label{sec:freegroupalgebrasdivision}

Our work can be regarded as an application of some techniques on the
existence of free group algebras in division rings. In this section, we gather together
the version of those results that we will use.

We begin with \cite[Theorem~3.2]{SanchezObtaininggraded}. It tells us
a way to obtain a free group algebra from a free algebra in case the division ring
is the Ore ring of fractions of a graded Ore domain.

\begin{theo}\label{theo:freegroupgradedOre}
Let $G$ be an orderable group and $k$ be a commutative ring. Let \linebreak
 $A=\bigoplus\limits_{g\in G} A_g$ be a $G$-graded $k$-algebra.
Let $X$ be a subset  of $A$
consisting of homogeneous elements where we denote by $g_x\in G$ the degree of $x\in X$,
 i.e. $x\in A_{g_x}$. Suppose that the following three conditions are satisfied.
\begin{enumerate}[\rm(1)]
\item There exists a strict total ordering $<$ of $G$ such that $(G,<)$ is an ordered group and
 $1<g_x$ for all $x\in X$.
\item The $k$-subalgebra of $A$ generated by $X$ is the free
$k$-algebra on $X$.
\item $A$ is a left Ore domain with left Ore ring of fractions $Q_{cl}(A)$.
\end{enumerate}
Then the $k$-subalgebra of $Q(A)$ generated by $\{1+x,\,
(1+x)^{-1}\}_{x\in X}$ is the free group $k$-algebra
on the set $\{1+x\}_{x\in X}$. \qed
\end{theo}

The next proposition is \cite[Proposition~2.5(4')]{SanchezObtainingI}.
It shows that the existence of a free group algebra in the graded ring
induced by a valuation on a division ring $D$, (under some circumstances) 
implies the existence of a free group
algebra in $D$.

\begin{prop}\label{prop:freeobjecthomogeneous}
Let $Z$ be a commutative ring and $R$ be a $Z$-algebra. Let 
$\upsilon\colon R\rightarrow \mathbb{Z}\cup\{\infty\}$ be a valuation. Let $X$
be a subset of elements of $R$ such that
the map $X\rightarrow \gr_\upsilon(R)$, $x\mapsto x+R_{>\upsilon(x)}$,
is injective.
\begin{enumerate}[\rm(1)]
\item The elements of $X$ are invertible in $R$.

\item The $Z_0$-subalgebra of $\gr_\upsilon(R)$ generated by
$\{x+R_{>\upsilon(x)},\, x^{-1}+R_{>\upsilon(x^{-1})}\}_{x\in X}$
is the free group $Z_0$-algebra on $\{x+R_{>\upsilon(x)}\}_{x\in
X}$.
\end{enumerate}
Then the $Z$-subalgebra of $R$ generated by $\{x,\,
x^{-1}\}_{x\in X}$ is the free group $Z$-algebra on $X$,
where $Z_0=Z_{\geq 0}/Z_{>0}$. \qed
\end{prop}

The next theorem is \cite[Theorem~3.2]{SanchezObtainingI}. It tells us that
sometimes, in order to find a free group algebra in division ring $D$,
it is enough to find a free algebra on the graded ring induced by
a valuation on $D$.

\begin{theo}\label{coro:divisionrings}
Let $D$ be a division ring with prime subring $Z$. Let  $\upsilon\colon
D\rightarrow \mathbb{R}\cup\{\infty\}$ be a nontrivial valuation.
 Let
$X$ be a subset of $D$ satisfying the following
three conditions.
\begin{enumerate}[\rm (1)]
\item The map $X\rightarrow\gr_\upsilon(D)$, $x\mapsto x+D_{>\upsilon(x)}$, is injective.
\item For each $x\in X$, $\upsilon(x)>0$.
\item The $Z_0$-subalgebra of $\gr_\upsilon(D)$ generated by the set
$\{x+D_{>\upsilon(x)}\}_{x\in X}$ is the free $Z_0$-algebra
on the set $\{x+D_{>\upsilon(x)}\}_{x\in X}$, where 
$Z_0= Z_{\geq 0}/Z_{>0}\subseteq D_0$.
\end{enumerate}
Then, for any central subfield $k$,  the
$k$-subalgebra of $D$ generated by $\{1+x,\, (1+x)^{-1}\}_{x\in
X}$ is the free group $k$-algebra on $\{1+x\}_{x\in X}$. \qed
\end{theo}


\section{Nilpotent Lie algebras with involutions}\label{sec:nilpotentLieinvolutions}

Let $k$ be a field and $L$ be a
Lie $k$-algebra. A $k$-linear map  $*\colon L\rightarrow L$ is a
$k$-\emph{involution} \label{kinvolutionLie}
if for all $x,y\in L$, $[x,y]^*=[y^*, x^*]$,
$x^{**}= x$. The main example of a $k$-involution in a Lie
$k$-algebra is what we call the \emph{principal involution}. It is
defined by $x\mapsto -x$ for all $x\in L$.

The \emph{Heisenberg Lie $k$-algebra} is the Lie $k$-algebra with
presentation
\begin{equation}\label{eq:Heisenberg}
H=\langle x,y\mid [[y,x],x]=[[y,x],y]=0 \rangle.
\end{equation}
The Heisenberg Lie $k$-algebra  can also be characterized as the
unique Lie $k$-algebra of dimension three such that $[H,H]$ has
dimension one and $[H,H]$ is contained in the center of $H$, see
\cite[Section~4.III]{JacobsonLiealgebras}.

Let $k$ be a field of characteristic different from $2$.
In this section, we first find all the $k$-involutions of $H$,
secondly we show that there are essentially three involutions on $H$, and then
that any nilpotent Lie $k$-algebra with involution 
contains a $k$-subalgebra isomorphic to $H$ invariant under the involution and such 
that the restriction of the involution to $H$ is one of those three ones.

\begin{lem}\label{lem:involutionHeisenbergalgebra}
Let $k$ be a field of characteristic different from two. Let \linebreak $H=\langle x,y \mid
[x,[y,x]]=[y,[y,x]]=0 \rangle$ be the Heisenberg Lie $k$-algebra and $z=[y,x]$.
Then any $k$-involution $\tau\colon H\rightarrow H$  is of one of
the following forms:
\begin{enumerate}[\rm(i)]
\item
$\left\{\begin{array}{l} \tau(x)= ax+by+cz \\
\tau(y)= dx-ay+fz \\ \tau(z)=z \end{array}\right.$ where $a,b,c,d,f\in k$
satisfy \ $\left\{\begin{array}{rrr}
 a^2 &      + & bd=1,  \\
   (a+1)c & + & bf=0 \\ dc & +& (1-a)f=0
\end{array}\right. $

\item $\left\{\begin{array}{l} \tau(x)= x+cz \\
\tau(y)= y+fz \\ \tau(z)= -z \end{array}\right.$ where $c,f\in k$.

\item $\left\{\begin{array}{l} \tau(x)= -x \\
\tau(y)= -y \\ \tau(z)= -z \end{array}\right.$
\end{enumerate}

\end{lem}

\begin{proof}
Let $*\colon H\rightarrow H$ be a $k$-involution on $H$. Note that
$Z(H)$, the center of $H$, is  the one-dimensional $k$-subspace
generated by $z=[y,x]$. Since $z^*\in Z(H)$ and $(z^*)^*=z$, we
obtain that $z^*=z$ or $z^*=-z$.

Suppose that $x^*=ax+by+cz$ and $y^*=dx+ey+fz$ where $a,b,c,d,e,f\in
k$.

\noindent{\underline{Case 1: $z^*=z$}.}
 Then
$z=[y,x]=[y,x]^*=[x^*,y^*]=[ax+by+cz,
dx+ey+fz]=[by,dx]+[ax,ey]=(bd-ae)z$. Thus
\begin{equation}\label{case1_*} bd-ae=1
\end{equation}
From
$x=(x^*)^*=(ax+by+cz)^*=a(ax+by+cz)+b(dx+ey+fz)+cz=(a^2+bd)x+(ab+be)y+(ac+bf+c)z$,
we get
\begin{eqnarray}
& a^2+bd=1, & \label{case1_a}\\
& b(a+e)=0, & \label{case1_b}\\
& ac+bf+c=0. & \label{case1_c}
\end{eqnarray}
From
$y=(y^*)^*=(dx+ey+fz)^*=d(ax+by+cz)+e(dx+ey+fz)+fz=d(a+e)x+(e^2+db)y+(cd+ef+f)z$,
we obtain
\begin{eqnarray}
& e^2+bd=1, & \label{case1_d}\\
& d(a+e)=0, & \label{case1_e}\\
& cd+ef+f=0. & \label{case1_f}
\end{eqnarray}
From \eqref{case1_a} and \eqref{case1_d}, we obtain that $a=\pm e$.

Suppose that $a=e$. Then \eqref{case1_a} and \eqref{case1_*} imply
that $a=e=0$. Thus. this case is contained in the case $a=-e$.

Suppose now that $a=-e$. Then \eqref{case1_a}, \eqref{case1_d} and
\eqref{case1_*} are in fact the same equation. Also the equations
\eqref{case1_b} and \eqref{case1_e} do not give any new information.
Thus \eqref{case1_c} and \eqref{case1_f} equal
\begin{eqnarray}\label{case1_coef}
  \left\{ \begin{array}{rrr} (a+1)c & + & bf=0 \\ dc & +& (1-a)f=0   \end{array}\right.
\end{eqnarray}
Observe that, by \eqref{case1_a}, $\det \begin{pmatrix} a+1 & b \\ d
& 1-a\end{pmatrix}=-a^2-bd+1=0$.

Therefore $x^*=ax+by+cz$, $y^*=dx-ay+fz$, $z^*=z$, where $a,b,c,d,f$
satisfy \eqref{case1_a} and \eqref{case1_coef}. Hence (i) is proved.

\noindent{\underline{Case 2: $z^*=-z$}.}
$-z=[y,x]^*=[x^*,y^*]=[ax+by+cz,
dx+ey+fz]=[by,dx]+[ax,ey]=(bd-ae)z$. Thus
\begin{equation}\label{case2_*} ae-bd=1
\end{equation}
From
$x=(x^*)^*=(ax+by+cz)^*=a(ax+by+cz)+b(dx+ey+fz)-cz=(a^2+bd)x+(ab+eb)y+(ac+bf-c)z$,
we get
\begin{eqnarray}
& a^2+bd=1, & \label{case2_a}\\
& b(a+e)=0, & \label{case2_b}\\
& ac+bf-c=0. & \label{case2_c}
\end{eqnarray}
From
$y=(y^*)^*=(dx+ey+fz)^*=d(ax+by+cz)+e(dx+ey+fz)-fz=d(a+e)x+(e^2+db)y+(cd+ef-f)z$,
we obtain
\begin{eqnarray}
& e^2+bd=1, & \label{case2_d}\\
& d(a+e)=0, & \label{case2_e}\\
& cd+ef-f=0. & \label{case2_f}
\end{eqnarray}
From \eqref{case2_a} and \eqref{case2_d}, we obtain that $a=\pm e$.

It is not possible that $a=-e$  because \eqref{case2_a} and
\eqref{case2_*} would imply that $1=-1$.

Suppose now that $a=e$. Then \eqref{case2_*}, \eqref{case2_a} and
\eqref{case2_d} imply that $a^2=1$. Hence $a=e=\pm 1$.  Now
\eqref{case2_b} and \eqref{case2_e} imply that $b=d=0$.

If $a=-1$,  we obtain, by \eqref{case2_c} and \eqref{case2_f}, that
$f=c=0$. Hence we obtain (iii), i.e. $x^*=-x$, $y^*=-y$  $z^*=-z$.

If $a=1$, \eqref{case2_c} and \eqref{case2_f} do not give any new
information. Hence we obtain (ii), i.e. $x^*=x+cz$, $y^*=y+fz$,
$z^*=-z$, where $c,f\in k$. 
\end{proof}

Let $k$ be a field.
Let $\tau,\eta\colon L\rightarrow L$ be two $k$-involutions of a Lie $k$-algebra $L$.
We say that $\tau$ is \emph{equivalent} to $\eta$ if there exists an isomophism of Lie $k$-algebras
$\varphi\colon L\rightarrow L$ such that $\varphi^{-1}\tau\varphi=\eta$.

\begin{lem}\label{lem:equivalentinvolutionHeisenbergalgebra}
Let $k$ be a field of characteristic different from two and   $H$ be the Heisenberg Lie $k$-algebra. Any
$k$-involution $\tau\colon H\rightarrow H$ is equivalent to one of the following involutions
$\eta\colon H\rightarrow H$. 
\begin{enumerate}[\rm (1)]
	\item The involution
	$\eta\colon H\rightarrow H$ defined by $\eta(x)=x,\, \eta(y)= -y,\,  \eta(z)= z.$
	More precisely, any $k$-involution in
Lemma~\ref{lem:involutionHeisenbergalgebra}{\rm(i)} is equivalent to $\eta$ just defined.

	\item The involution $\eta\colon H\rightarrow H$ defined by $\eta(x)= x,\, \eta(y)= y,\, \eta(z)= -z$.
	More precisely, any $k$-involution in  Lemma~\ref{lem:involutionHeisenbergalgebra}{\rm(ii)}
is equivalent to  $\eta$ just defined.

	\item The principal involution $\eta\colon H\rightarrow H$ defined by $\eta(x)=-x,\, \eta(y)=-y,\, \eta(z)=-z$. 
	\end{enumerate}
Furthermore, we exhibit explicit isomorphisms  $\varphi\colon
H\rightarrow H$ which prove that $\varphi^{-1}\tau\varphi=\eta$  where  $\tau$ is any involution in
Lemma~\ref{lem:involutionHeisenbergalgebra}{\rm(i)} and {\rm(ii)}.
\end{lem}

\begin{proof}
Clearly the involution of Lemma~\ref{lem:involutionHeisenbergalgebra}(iii) is the same as the one in (3).

First we prove (2). Let $f\mapsto f^*$ be any involution in Lemma~\ref{lem:involutionHeisenbergalgebra}(ii).
 Suppose that $c,f\in k$ and that $x^*=x+cz$,
 $y^*=y+fz$ and $z^*=-z$. Define $X=\frac{1}{2}(x+x^*)=\frac{1}{2}(2x+cz)$,
 $Y=\frac{1}{2}(y+y^*)=\frac{1}{2}(2y+fz)$ and $Z=z$.
 Note that $X,Y, Z$ form a $k$-basis of $H$ and that $[Y,X]=[y,x]=z=Z$. Thus
 there exists an isomorphism  $\varphi\colon H\rightarrow H$ sending $x\mapsto X$, $y\mapsto Y$ and
 $z\mapsto Z$. Moreover $X^*=X$, $Y^*=Y$ and $Z^*=-Z$, as desired.

Now we prove (1). Let $h\mapsto h^*$ be any involution of 
Lemma~\ref{lem:involutionHeisenbergalgebra}(i).
 Let $a,b,c,d,f\in k$ satisfying the
conditions in Lemma~\ref{lem:involutionHeisenbergalgebra}(i). Hence $x^*=ax+by+cz $,
$y^*=dx-ay+fz$,  $z^*=z$. We consider three cases:
\begin{enumerate}[(I)]
\item $b\neq 0$
\item $d\neq 0$
\item $b=d=0$.
\end{enumerate}

(I) Suppose $b\neq 0$. Define $X=\frac{1+a}{2}x+\frac{b}{2}y+\frac{c}{2}z$, $Y=\frac{1-a}{2}x-\frac{b}{2}y-\frac{c}{2}z$ and $Z=-\frac{b}{2}z$.
Note that $X,Y,Z$ is a $k$-basis of $H$ and that $[Y,X]=Z.$ Thus
there exists an isomorphism of $\varphi\colon H\rightarrow H$ sending  $x\mapsto X$, $y\mapsto
Y$, $z\mapsto Z$. Note that $X^*=X$, $Y^*=-Y$ and $Z^*=Z$, as
desired.

(II) Suppose now that $d\neq 0$. Define $X=\frac{d}{2}x+\frac{1-a}{2}y+\frac{f}{2}z$, $Y=\frac{d}{2}x-\frac{a+1}{2}y+\frac{f}{2}z$ and
$Z=-\frac{d}{2}z$. Note that $X,Y,Z$ is a $k$-basis of $H$ and that $[Y,X]=Z$.
Thus there exists an isomorphism  $\varphi\colon H\rightarrow H$ given by $x\mapsto X$,
$y\mapsto Y$, $z\mapsto Z$. Note that $X^*=X$, $Y^*=-Y$ and $Z^*=Z$,
as desired.

(III) Suppose that $b=d=0$. Then $a^2=1$ and either $c=0$ or $f=0$.
In both cases define $X=-\frac{1+a}{2}x+\frac{1-a}{2}y+\frac{f-c}{2}z$, $Y=\frac{a-1}{2}x+\frac{1+a}{2}y+\frac{c-f}{2}z$, $Z=-az$. It
is not difficult to show that $X,Y,Z$ is a $k$-basis of $H$ and that
$[Y,X]=Z$. Thus there exists an isomorphism  $\varphi\colon H\rightarrow H$ given by
$x\mapsto X$, $y\mapsto Y$, $z\mapsto Z$. Note that $X^*=X$, $Y^*=-Y$
and $Z^*=Z$, as desired.
\end{proof}

The following two results are the Lie algebra version of 
\cite[Lemma~2.3, Proposition~2.4]{FerreiraGoncalvesSanchezFreegroupssymmetric}. 
The proofs are analogous to the ones given there for groups.

\begin{lem}\label{lem:involutionclass2}
Let $k$ be a field of characteristic different from two, and $L$ be a finitely generated nilpotent
Lie $k$-algebra of class $2$ with involution $*\colon L\rightarrow L$, $f\mapsto f^*$.
Then $L$ contains a $*$-invariant Heisenberg Lie $k$-subalgebra $H$ and the
restriction to $H$ is one of the involutions in Lemma~\ref{lem:equivalentinvolutionHeisenbergalgebra}. More precisely,
there exist $x,y\in L$ such that $[y,x]\neq0$,
$[y,[y,x]]=[x,[y,x]]=0$, and either
$x^*=x,\ y^*=-y$, or $x^*=x,\ y^*=y$, or $x^*=-x,\ y^*=-y$.
\end{lem}

\begin{proof}
Let $C$ denote the center of $L$. It follows from the nilpotency class of $L$
that $L/C$ is a finitely generated torsion-free abelian Lie $k$-algebra and the involution
induces an automorphism of  $k$-vector spaces $\varphi\colon L/C\rightarrow L/C$, $f+C\mapsto f^*+C$.
Notice that $L/C$ has dimension at least two because $L$ is not abelian. Since $\varphi^2$ is
the identity, $\varphi$ is diagonalizable. There exist $u_1,\dotsc,u_n\in L$ such 
that $\{u_1+C,\dotsc,u_n+C\}$ is a basis of $L/C$ consisting of eigenvectors with eigenvalues 
$\pm 1$. Since $L$ is not abelian, we may suppose that $[u_1,u_2]\neq 0$.
Hence there exist $z_1,z_2\in C$ such that $u_i^*=\varepsilon_iu_i+z_i$,
where $\varepsilon_i\in\{1,-1\}$, for $i=1,2$.

Suppose that $u_1^*=-u_1+z_1$ and $u_2^*=u_2+z_2$. 
Let $H$ be the subalgebra with basis $x=\frac{1}{2}(u_1- u_1^*)=u_1-\frac{1}{2}z_1$,
$y=\frac{1}{2}(u_2+u_2^*)=u_2+\frac{1}{2}z_2$ and $z=[y,x]$. Now proceed as in the previous case.

Suppose that $u_1^*=u_1+z_1$ and $u_2^*=-u_2+z_2$. 
Let $H$ be the subalgebra with basis $x=\frac{1}{2}(u_2- u_2^*)=u_2-\frac{1}{2}z_2$,
$y=\frac{1}{2}(u_1+u_1^*)=u_2+\frac{1}{2}z_2$ and $z=[y,x]$.
Clearly $z$ commutes with $x$ and $y$ because it is an element of $C$.
Now $x^*=-x$, $y^*=y$ and
$z^*=[y,x]^*=[x^*,y^*]=[-x,y]=z$. Thus 
$*$, when restricted to $H$, is the involution (1) in Lemma~\ref{lem:involutionHeisenbergalgebra}.

Suppose that $u_1^*=u_1+z_1$ and $u_2^*=u_2+z_2$. Let $H$
be the subalgebra with basis $x=\frac{1}{2}(u_1+ u_1^*)=u_1+\frac{1}{2}z_1$,
$y=\frac{1}{2}(u_2+u_2^*)=u_2+\frac{1}{2}z_2$ and $z=[y,x]$. 
Clearly $z$ commutes with $x$ and $y$ because it is an element of $C$. Now $x^*=x$, $y^*=y$ and
$z^*=[y,x]^*=[x^*,y^*]=[x,y]=-z$.
Thus $*$, when restricted to $H$, is the involution (2) in Lemma~\ref{lem:involutionHeisenbergalgebra}.

Suppose that $u_1^*=-u_1+z_1$ and $u_2^*=-u_2+z_2$. Let $H$
be the subalgebra with basis $x=\frac{1}{2}(u_1- u_1^*)=u_1-\frac{1}{2}z_1$,
$y=\frac{1}{2}(u_2-u_2^*)=u_2-\frac{1}{2}z_2$ and $z=[y,x]$. Clearly $z$ commutes with $x$ and $y$ because it is an element of $C$. Now $x^*=-x$, $y^*=-y$ and
$z^*=[y,x]^*=[x^*,y^*]=[-x,-y]=-z$. Thus $*$,
when restricted to $H$, is the involution (3) in Lemma~\ref{lem:involutionHeisenbergalgebra}
\end{proof}

\begin{prop}\label{prop:involutionnilpotent}
Let $k$ be a field of characteristic different from two, and $L$ be a non-abelian nilpotent
Lie $k$-algebra with involution $*\colon L\rightarrow L$, $f\mapsto f^*$.
Then $L$ contains a $*$-invariant Heisenberg Lie $k$-subalgebra $H$ such
that the restriction to $H$ is one of the involutions in Lemma~\ref{lem:equivalentinvolutionHeisenbergalgebra}. More precisely,
there exist $x,y\in L$ such that $[y,x]\neq0$,
$[y,[y,x]]=[x,[y,x]]=0$, and either
$x^*=x,\ y^*=-y$, or $x^*=x,\ y^*=y$, or $x^*=-x,\ y^*=-y$.
\end{prop}

\begin{proof}
By taking the subalgebra of $L$ generated by two noncommuting elements and their
images by $*$, we can assume that $L$ is finitely generated.

We shall argue by induction on the nilpotency class $c$ of $L$; the case $c=2$
having been dealt with in Lemma~\ref{lem:involutionclass2}.

Suppose that $c>2$ and let $C$ denote the center of $L$. Then $L/C$
is a nonabelian finitely generated nilpotent Lie algebra of class $c-1$, with
an involution induced by $*$. By the induction hypothesis, there exist $x,y\in L$
such that $\{x+C,y+C\}$ generate a $*$-invariant Heisenberg Lie subalgebra of $L/C$.
Moreover $x^*+C=\varepsilon x+C$ and $y^*+C=\eta y+C$ with $\varepsilon,\eta\in\{1,-1\}$
and $z=[y,x]\notin C$, $[y,z],[x,z]\in C$. It follows that $M$, the
subalgebra of $L$ generated by $\{x,y,C\}$, is a $*$-invariant subalgebra of $L$
of nilpotency class $\leq3$. 

If $M$ has class 2, then the result follows from Lemma~\ref{lem:involutionclass2}.

Suppose that $M$ has class 3. Then $[x,z]\neq 0$ or $[y,z]\neq 0$. Say $[x,z]\neq 0$.
We shall show that the $k$-subalgebra $K$ generated by $\{x,x^*,z,z^*\}$ is a $*$-invariant
subalgebra of $L$ of class $2$. It will be enough to show that $[\alpha,\beta]$
lies in the center of $K$ for all $\alpha,\beta\in \{x,x^*,z,z^*\}$. For each $n\geq 1$,
let $\gamma_n(M)$ denote the $n$-th term in the lower central series of $M$. Now,
$z=[y,x]\in\gamma_2(M)$, so $z^*\in\gamma_2(M)$, because the terms in the lower central
series are fully invariant subgroups of $G$. It follows that for every $\alpha\in K$ and
$\beta\in \{z,z^*\}$ we haver $[\alpha,\beta]\in\gamma_3(M)$, which is a central subalgebra
of $M$, hence $[\alpha,\beta]$ is central in $K$. Finally, that $[x,x^*]=0$ follows
from the fact that $x^*-x\in C$. So $K$ is indeed a $*$-invariant subalgebra of $L$ of class $2$.
Hence Lemma~\ref{lem:involutionclass2} applies to it.
\end{proof}

\section{Free group algebras in the Ore ring of fractions of universal enveloping  algebras
of nilpotent Lie algebras}\label{sec:nilpotentOre}

Let $k$ be a field of characteristic zero and $H$ be the Heisenberg Lie $k$-algebra.
In this section we find different free (group) $k$-subalgebras in $\mathfrak{D}(H)$,
the Ore ring of fraction of the universal enveloping algebra $U(H)$ of $H$.
For that, our main tool is the result by G. Cauchon~\cite[Th\'eor\`eme]{Cauchoncorps}. 
The technique to obtain suitable free algebras from the paper of Cauchon was developed in
\cite[Section~3]{FerreiraGoncalvesSanchez2}. In some cases, we then use Theorem~\ref{theo:freegroupgradedOre}
to obtain free group algebras in $\mathfrak{D}(H)$. If $L$ is a nilpotent Lie $k$-algebra,  
applying Proposition~\ref{prop:involutionnilpotent}, 
the foregoing implies the existence of free group algebras in $\mathfrak{D}(L)$, the Ore ring of fractions
of the universal enveloping algebra $U(L)$ of $L$.

\medskip

Let $k$ be a field of characteristic different from two. Let $K=k(t)$ be the field of fractions of the
polynomial ring $k[t]$ in the variable $t$. Let $\sigma$ be a
$k$-automorphism of $K$ of \emph{infinite order}. We will consider the skew polynomial ring
$K[p;\sigma]$. The elements of $K[p;\sigma]$ are ``right
polynomials'' of the form $\sum_{i=0}^n p^ia_i$, where the
coeficients $a_i$ are in $K$. The multiplication is determined by
$$ap=p\sigma(a) \quad \textrm{for all } a\in K.$$
It is known that $K[p;\sigma]$ is a noetherian domain and therefore
it has an Ore division ring of fractions $D=K(p;\sigma)$.

Since $\sigma$ is an automorphism of $K$,
$\sigma(t)=\frac{at+b}{ct+d}$ where $M=\left(\begin{smallmatrix} a & b \\
c & d
\end{smallmatrix}\right)\in GL_2(k)$ defines a homography $h$ of the
projective line $\Delta=\mathbb{P}_1(k)=k\cup\{\infty\}$, $h\colon
\Delta \rightarrow \Delta,$ $z\mapsto h(z)=\frac{az+b}{cz+d}$.

We denote by $\mathcal{H}=\{h^n\mid n\in\mathbb{Z} \}$ the subgroup
of the projective linear group $PGL_2(k)$ generated by $h$. The
group $\mathcal{H}$ acts on $\Delta$. If $z\in\Delta$, we denote by
$\mathcal{H}\cdot z=\{h^n(z)\mid n\in\mathbb{Z}\}$ the orbit of $z$
under the action of $\mathcal{H}$.

\begin{theo}[Cauchon's Theorem]
Let $\alpha$ and $\beta$ be two elements of $k$ such that the orbits
$\mathcal{H}\cdot \alpha$ and $\mathcal{H}\cdot \beta$ are infinite
and different. Let $s$ and $u$ be the two elements of $D$ defined by
$$s=(t-\alpha)(t-\beta)^{-1}, \quad u=(1-p)(1+p)^{-1}.$$
If the characteristic of $k$ is different from $2$, then the
$k$-subalgebra $\Omega$ of $D$ generated by $\xi=s$,
$\eta=usu^{-1}$, $\xi^{-1}$ and $\eta^{-1}$ is the free group
$k$-algebra on the set $\{\xi,\eta\}$. \qed
\end{theo}

We will need the following consequence of Cauchon's Theorem.

\begin{prop}\label{prop:freealgebrainWeyl}
Let $k$ be a field of  characteristic zero and $K=k(t)$ be the field of
fractions of the polynomial ring $k[t]$. Let $\sigma\colon
K\rightarrow K$ be the automorphism of $k$-algebras determined by
$\sigma(t)=t-1$. Consider the skew polynomial ring $K[p;\sigma]$ and
its Ore division ring of fractions $K(p;\sigma)$.
Set $s=(t-\frac{5}{6})(t-\frac{1}{6})^{-1}$,  $u=(1-p^2)(1+p^2)^{-1}$ and
$u_1=(1-p^3)(1+p^3)^{-1}$. The following hold true.

\begin{enumerate}[\rm(1)]
\item  The
$k$-subalgebra of $K(p;\sigma)$ generated by $\{s,\, s^{-1},\ usu^{-1},\, us^{-1}u^{-1}\}$ is the free
group $k$-algebra on the set
$\{s,\ usu^{-1}\}$.

\item  The $k$-subalgebra of $K(p;\sigma)$ generated by $\{s+s^{-1},\
u(s+s^{-1})u^{-1}\}$ is the free $k$-algebra on the set
$\{s+s^{-1},\ u(s+s^{-1})u^{-1}\}$.

\item The $k$-subalgebra of $K(p;\sigma)$ generated by $\{s+s^{-1},\
u_1(s+s^{-1})u_1^{-1}\}$ is the free $k$-algebra on the set
$\{s+s^{-1},\ u_1(s+s^{-1})u_1^{-1}\}$.
\end{enumerate}
\end{prop}

\begin{proof}
We will apply Cauchon's Theorem to the skew polynomial ring
$K[p^2;\sigma^2]$, where $\sigma^2\colon K\rightarrow K$ is given by
$\sigma^2(t)=t-2$.

Let $\alpha=\frac{5}{6}\in k$ and $\beta=\frac{1}{6}$. Let
$\mathcal{H}$ be defined as above. Consider the orbits
$\mathcal{H}\cdot\alpha=\{\frac{5}{6}-2n\mid n\in\mathbb{Z}\}$,
$\mathcal{H}\cdot\beta=\{\frac{1}{6}-2n\mid n\in\mathbb{Z}\}$ which
are infinite and different.

Then, by Cauchon's Theorem, $s=(t-\alpha)(t-\beta)^{-1}$ and
$u=(1-p^2)(1+p^2)^{-1}$ are such that the $k$-algebra generated by
$\xi=s, \eta=usu^{-1}$, $\xi^{-1}$ and $\eta^{-1}$ is the free group
$k$-algebra on the free generators $\{\xi,\eta\}$. Thus (1) is proved.

By Corollary~\ref{coro:freeinsidegroupring}, the $k$-algebra
generated by  $\{s+s^{-1},\ u(s+s^{-1})u^{-1}\}$ is the free
$k$-algebra on the set $\{s+s^{-1},\ u(s+s^{-1})u^{-1}\}$. Thus (2) is proved.

In order to prove (3), apply Cauchon's Theorem to the skew polynomial ring
$K[p^3;\sigma^3]$, where $\sigma^3\colon K\rightarrow K$ is given by
$\sigma^3(t)=t-3$. Then proceed as in (1) and (2).
\end{proof}

The following lemma is well known. For example, it appears in \cite[Section~3.3]{FerreiraGoncalvesSanchez2}.

\begin{lem}\label{lem:specializationtoWeyl}
Let $k$ be a field of characteristic zero. Let $K=k(t)$ be the field of fractions of the polynomial
ring $k[t]$ and $\sigma\colon K\rightarrow K$ be the automorphism of $k$-algebras determined by
$\sigma(t)=t-1$. Consider the skew polynomial ring $K[p;\sigma]$ and its Ore ring of fractions
$K(p;\sigma)$. Let $H=\langle x,y\mid [[y,x],x]=[[y,x],y]=0 \rangle$ be the Heisenberg Lie $k$-algebra, set $z=[y,x]$
and consider the universal enveloping algebra $U(H)$ of $H$.
. The following hold true
\begin{enumerate}[\rm(1)]
	\item Set $I=U(H)(z-1)$, the ideal of $U(H)$
	generated by $z-1$. The set $\mathfrak{S}=U(H)\setminus I$ is a left Ore subset of $U(H)$.
	\item There exists a surjective $k$-algebra homomorphism $\Phi\colon \mathfrak{S}^{-1}U(H)\rightarrow K(p;\sigma)$
	such that $\Phi(y)=p$, $\Phi(x)=p^{-1}t$ and $\Phi(z)=1$.
	\end{enumerate}
\end{lem}

\begin{proof}
First note that $$p(p^{-1}t)-(p^{-1}t)p=t-p^{-1}p(t-1)=1.$$
Hence there exists a $k$-algebra homomorphism $\Phi\colon U(H)\rightarrow K(p;\sigma)$
such that $\Phi(y)=p$, $\Phi(x)=p^{-1}t$ and $\Phi(z)=1$. The ideal $I$ is clearly contained in the kernel of $\Phi$.
Now note that $U(H)/I$ is the first Weyl algebra, which is a simple $k$-algebra. Thus $I$ is the kernel of $\Phi$. 
The subset $\mathfrak{S}$ is an Ore subset of $U(H)$ by \cite[Lemma~13]{Lichtmanfreeuniversalenveloping}.
By the universal property of the Ore localization $\Phi$ can be extended to a $k$-algebra
homomorphism $\Phi\colon \mathfrak{S}^{-1}U(H)\rightarrow K(p;\sigma)$. Note that $\mathfrak{S}^{-1}U(H)$
is a local ring with maximal ideal $\mathfrak{S}^{-1}I$. It induces an embedding of
division rings $\mathfrak{S}^{-1}U(H)/\mathfrak{S}^{-1}I\rightarrow K(p;\sigma)$. Now
$\Phi$ is surjective because $\Phi(yx)=t$ and $\Phi(y)=p$.
\end{proof}

The next result is \cite[Corollary~3.2]{FerreiraGoncalvesSanchez2}. It will
allow us to obtain  free algebras generated by symmetric elements 
from  free group algebras.  
\begin{lem}\label{coro:freeinsidegroupring}
Let $G$ be the free group on the set of two elements $\{x,y\}$. Let
$k$ be a field and consider the group algebra $k[G]$. Then the
$k$-algebra generated by $x+x^{-1}$ and $y+y^{-1}$ inside $k[G]$ is
free on $\{x+x^{-1}, y+y^{-1}\}$. \qed
\end{lem}

Now we are ready to present the main result of this section. It will be used throughout the paper. 
Parts (1),(2) and (3) in Sections~\ref{sec:residuallynilpotent} and~\ref{sec:Ore} while
parts (4),(5) in Section~\ref{sec:Heisenberggroup}.

\begin{theo}\label{theo:freegroupHeisenberg}
Let $k$ be a field of  characteristic zero. Let $H$ be the
Heisenberg Lie $k$-algebra.
Let $U(H)$ be the universal enveloping algebra of $H$ and
$\mathfrak{D}(H)$ be the Ore division ring of fractions of $U(H)$.
Set $z=[y,x]$, $V=\frac{1}{2}(xy+yx)$, and consider the following
elments of $\mathfrak{D}(H)$:
$$S=(V-\frac{1}{3}z)(V+\frac{1}{3}z)^{-1},$$
$$T=(z+y^2)^{-1}(z-y^2)S(z+y^2)(z-y^2)^{-1},$$
$$S_1=z^{-1}\Big((V-\frac{1}{3}z)(V+\frac{1}{3}z)^{-1}+ (V-\frac{1}{3}z)^{-1}(V+\frac{1}{3}z)\Big) z^{-1},$$
$$S_2=z\Big((V-\frac{1}{3}z)(V+\frac{1}{3}z)^{-1}+ (V-\frac{1}{3}z)^{-1}(V+\frac{1}{3}z)\Big) z,$$
$$T_1=(z+y^2)^{-1}(z-y^2)S_1(z+y^2)(z-y^2)^{-1},$$
$$T_2=(z^2+y^3)^{-1}(z^2-y^3)S_1(z^2+y^3)(z^2-y^3)^{-1},$$
$$T_3=(z+y^2)^{-1}(z-y^2)S_2(z+y^2)(z-y^2)^{-1},$$
$$T_4=(z^2+y^3)^{-1}(z^2-y^3)S_2(z^2+y^3)(z^2-y^3)^{-1}.$$

The following hold true.
\begin{enumerate}[\rm (1)]
\item The $k$-subalgebra of $\mathfrak{D}(H)$ generated by $\{S,\, S^{-1},\, T,\, T^{-1}\}$
 is the free group $k$-algebra on the set $\{S,T\}$.
\item
\begin{enumerate}[\rm (a)]
    \item The elements $S_1$, $S_1^2$, $T_1$ and $T_1^2$ are symmetric with respect to the involutions in 
 Lemma~\ref{lem:equivalentinvolutionHeisenbergalgebra}~(2) and (3).
    \item The $k$-subalgebra of $\mathfrak{D}(H)$ generated by $\{S_1,T_1\}$ is the free $k$-algebra on the 
		set $\{S_1,T_1\}$.
    \item The $k$-subalgebra of $\mathfrak{D}(H)$ generated by $\{1+S_1,(1+S_1)^{-1}, 1+T_1,(1+T_1)^{-1}\}$
    is the free group $k$-algebra on the set $\{1+S_1,1+T_1\}$.
		\item The $k$ subalgebra of $\mathfrak{D}(H)$ generated by $\{S_1^2,T_1^2\}$ is the free $k$-algebra on the 
		set $\{S_1^2,T_1^2\}$.
    \item The $k$-subalgebra of $\mathfrak{D}(H)$ generated by $\{1+S_1^2,(1+S_1^2)^{-1}, 1+T_1^2,(1+T_1^2)^{-1}\}$
    is the free group $k$-algebra on the set $\{1+S_1,1+T_1\}$.
\end{enumerate}
\item \begin{enumerate}[\rm (a)]
    \item The elements $S_1$, $S_1^2$, $T_2$ and $T_2^2$ are symmetric with respect to the involution in Lemma~\ref{lem:equivalentinvolutionHeisenbergalgebra}~(1).
\item The $k$-subalgebra of $\mathfrak{D}(H)$ generated by $\{S_1,T_2\}$ is the free $k$-algebra 
on the set $\{S_1,T_2\}$.
\item The $k$-subalgebra of $\mathfrak{D}(H)$ generated by $\{1+S_1,(1+S_1)^{-1}, 1+T_2,(1+T_2)^{-1}\}$
    is the free group $k$-algebra on the set $\{1+S_1,1+T_2\}$.
		\item The $k$-subalgebra of $\mathfrak{D}(H)$ generated by $\{S_1^2,T_2^2\}$ is the free $k$-algebra on the 
		set $\{S_1^2,T_2^2\}$.
    \item The $k$-subalgebra of $\mathfrak{D}(H)$ generated by $\{1+S_1^2,(1+S_1^2)^{-1}, 1+T_2^2,(1+T_2^2)^{-1}\}$
    is the free group $k$-algebra on the set $\{1+S_1^2,1+T_2^2\}$.
\end{enumerate}
\item
\begin{enumerate}[\rm (a)]
\item The elements $S_2$, $S_2^2$, $T_3$ and $T_3^2$ are symmetric with respect to the involutions in 
 Lemma~\ref{lem:equivalentinvolutionHeisenbergalgebra}~(2) and (3).
\item The $k$-subalgebra of $\mathfrak{D}(H)$ generated by $\{S_2^2,T_3^2\}$ is the free $k$-algebra on the 
		set $\{S_2^2,T_3^2\}$.
\end{enumerate}
\item \begin{enumerate}[\rm (a)]
 \item The elements $S_2$, $S_2^2$, $T_4$ and $T_4^2$ are symmetric with respect to the involution in 
 Lemma~\ref{lem:equivalentinvolutionHeisenbergalgebra}~(1).
\item The $k$-subalgebra of $\mathfrak{D}(H)$ generated by $\{S_2^2,T_4^2\}$ is the free $k$-algebra on the 
		set $\{S_2^2,T_4^2\}$.
\end{enumerate}

\end{enumerate}
\end{theo}

\begin{proof}
Consider the surjective $k$-algebra homomorphism $\Phi\colon
\mathfrak{S}^{-1}U(H)\rightarrow K(p;\sigma)$ given in Lemma~\ref{lem:specializationtoWeyl}.
Then $\Phi(y)=p$, $\Phi(x)=p^{-1}t$ and $\Phi(z)=1$. 

Recall that in
$K(p;\sigma)$, we have and $tp=p(t-1)$. 
Thus $\Phi(V)=\Phi(\frac{1}{2}(xy+yx))=\frac{1}{2}(p^{-1}tp+pp^{-1}t)=\frac{1}{2}(t-1+t)=t-\frac{1}{2}$.
\begin{eqnarray*}
\Phi(V-\frac{1}{3}z)=t-\frac{1}{2}-\frac{1}{3}=t-\frac{5}{6}, & & \Phi(V+\frac{1}{3}z)=t-\frac{1}{2}+\frac{1}{3}=t-\frac{1}{6}, \\
\Phi(z+y^2)=1+p^2, & &  \Phi(z-y^2)=1-p^2,\\
\Phi(z^2+y^3)=1+p^3, & & \Phi(z^2-y^3)=1-p^3.
\end{eqnarray*}
Hence the elements $V-\frac{1}{3}z,\, V+\frac{1}{3}z,\, z+y^2,\, z-y^2,\,  z^2+y^3,\, z^2-y^3$ are invertible
in $\mathfrak{S}^{-1}U(H)$.

 Thus
$S,S^{-1},T,T^{-1}, S_1,S_1^{-1}, T_1,T_1^{-1}, T_2,T_2^{-1}\in \mathfrak{S}^{-1}U(H)$. 
Moreover, following the notation of
Proposition~\ref{prop:freealgebrainWeyl},
$$\Phi(S)=(t-\frac{5}{6})(t-\frac{1}{6})^{-1}=s,$$
$$\Phi(S_1)=(t-\frac{5}{6})(t-\frac{1}{6})^{-1}+(t-\frac{1}{6})(t-\frac{5}{6})^{-1}=s+s^{-1},$$
$$\Phi(S_2)=(t-\frac{5}{6})(t-\frac{1}{6})^{-1}+(t-\frac{1}{6})(t-\frac{5}{6})^{-1}=s+s^{-1},$$
$$\Phi((z+y^2)^{-1}(z-y^2))=(1+p^2)^{-1}(1-p^2)=(1-p^2)(1+p^2)^{-1}=u,$$
$$\Phi((z+y^2)(z-y^2)^{-1})=(1+p^2)(1-p^2)^{-1}=u^{-1},$$ 
$$\Phi((z^2+y^3)(z^2-y^3)^{-1})=(1+p^3)(1-p^3)^{-1}={u}_1^{-1}.$$
 
Hence
$\Phi(T)=usu^{-1}$,
 $\Phi(T_1)=u(s+s^{-1})u^{-1}$, $\Phi(T_2)=u_1(s+s^{-1})u_1^{-1}$, 
$\Phi(T_3)=u(s+s^{-1})u^{-1}$ and $\Phi(T_4)=u_1(s+s^{-1})u_1^{-1}$.

\medskip

We proceed to show that the elements in statements (1),(2),(3),(4),(5) generate free (group) algebras.
That they are symmetric will be proved below.

\medskip

(1)  By Proposition~\ref{prop:freealgebrainWeyl}(1), the set 
 $\{s,\,s^{-1},\, usu^{-1},\,
us^{-1}u^{-1}\}$ generates a free group $k$-algebra. Therefore, the
$k$-subalgebra of $\mathfrak{S}^{-1}U(H)$  generated by  $\{S,\, S^{-1},\, T,\, T^{-1}\}$
is the free group $k$-algebra on the
set $\{S,\ T\}$.

(2) By Proposition~\ref{prop:freealgebrainWeyl}(2), the set
$\{s+s^{-1},u(s+s^{-1})u^{-1}\}$ are the free generators of a free $k$-algebra. Therefore the $k$-subagebra generated by $\{S_1,T_1\}$
is the free $k$-algebra on $\{S_1,T_1\}$. This implies that
the $k$-subalgebra generated by $\{S_1^2,T_1^2\}$ is the free $k$-algebra on $\{S_1^2,T_1^2\}$.

Consider $H$ as a $\mathbb{Z}$-graded Lie $k$-algebra as in Example~\ref{ex:gradedLie}(a). Then $U(H)$ is graded according to
Lemma~\ref{lem:gradeduniversalenveloping}. The $k$-algebra $U(H)$ is an Ore domain. 
Recall that given a
$\mathbb{Z}$-graded $k$-algebra 
which is an Ore domain, localizing at the set of nonzero homogeneous elements yields  a graded division ring.
Thus if we localize at the set $\mathcal{H}$ of homogeneous elements of $U(H)$, we get that
$\mathcal{H}^{-1}U(H)$ is a graded division ring. 
Notice that $z,\, V-\frac{1}{3}z,\, V+\frac{1}{3}z$ are homogeneous of degree $-2$. Therefore
$S_1$ is homogeneous of degree $4$ and $S_1^2$
is homogeneous of degreee $8$. Notice that $z+y^2,\, z-y^2$ are homogeneous of degree $-2$. Therefore $T_1$ is homogeneous of degree $4$
and $T_1^2$ is homogeneous of degree $8$.
By Theorem~\ref{theo:freegroupgradedOre}, the $k$-subalgebra generated by the set $\{1+S_1,\, (1+S_1)^{-1},\, 1+T_1,\, (1+T_1)^{-1}\}$
is the free group $k$-algebra on the set $\{1+S_1,1+T_1\}$. Also, by  Theorem~\ref{theo:freegroupgradedOre}, 
the $k$-subalgebra generated by the set $\{1+S_1^2,\, (1+S_1^2)^{-1},\, 1+T_1^2,\, (1+T_1^2)^{-1}\}$
is the free group $k$-algebra on the set $\{1+S_1^2,1+T_1^2\}$, as desired.

(3) By Proposition~\ref{prop:freealgebrainWeyl}(3), the set
$\{s+s^{-1},u_1(s+s^{-1})u_1^{-1}\}$ are the free generators of a free $k$-algebra. Therefore the $k$-subagebra generated by $\{S_1,T_2\}$
is the free $k$-algebra on $\{S_1,T_2\}$.  This implies that
the $k$-subalgebra generated by $\{S_1^2,T_2^2\}$ is the free $k$-algebra on $\{S_1^2,T_2^2\}$.

Consider $H$ as a $\mathbb{Z}$-graded Lie $k$-algebra as in Example~\ref{ex:gradedLie}(b). Then $U(H)$ is graded according to
Lemma~\ref{lem:gradeduniversalenveloping}. The $k$-algebra $U(H)$ is an Ore domain. 
Recall that given a
$\mathbb{Z}$-graded $k$-algebra 
which is an Ore domain, localizing at the set of nonzero homogeneous elements yields  a graded division ring.
Thus if we localize at the set $\mathcal{H}$ of homogeneous elements of $U(H)$, we get that
$\mathcal{H}^{-1}U(H)$ is a graded division ring. 
Notice that $z,\, V-\frac{1}{3}z,\, V+\frac{1}{3}z$ are homogeneous of degree $-3$. Therefore
$S_1$ is homogeneous of degree $6$ and $S_1^2$ is homogeneous
of degree $12$. Notice that $z^2+y^3,\, z^2-y^3$ are homogeneous of degree $-6$. 
Therefore $T_2$ is homogeneous of degree $6$ and $T_2^2$ is homogeneous of degree $12$.
By Theorem~\ref{theo:freegroupgradedOre}, the $k$-subalgebra generated by the set $\{1+S_1,\, (1+S_1)^{-1},\, 1+T_2,\, (1+T_2)^{-1}\}$
is the free group $k$-algebra on the set $\{1+S_1,1+T_2\}$. Also, by 
Theorem~\ref{theo:freegroupgradedOre}, the $k$-subalgebra generated by the set $\{1+S_1^2,\, (1+S_1^2)^{-1},\, 1+T_2^2,\, (1+T_2^2)^{-1}\}$
is the free group $k$-algebra on the set $\{1+S_1^2,1+T_2^2\}$, as desired.

(4) By Proposition~\ref{prop:freealgebrainWeyl}(2), the set
$\{s+s^{-1},u(s+s^{-1})u^{-1}\}$ are the free generators of a free $k$-algebra. Therefore the $k$-subagebra generated 
by $\{S_2,T_3\}$
is the free $k$-algebra on $\{S_2,T_3\}$. This implies that
the $k$-subalgebra generated by $\{S_2^2,T_3^2\}$ is the free $k$-algebra on $\{S_2^2,T_3^2\}$.

(5) By Proposition~\ref{prop:freealgebrainWeyl}(3), the set
$\{s+s^{-1},u_1(s+s^{-1})u_1^{-1}\}$ are the free generators of a free $k$-algebra. 
Therefore the $k$-subagebra generated by $\{S_2,T_4\}$
is the free $k$-algebra on $\{S_2,T_4\}$.  This implies that
the $k$-subalgebra generated by $\{S_2^2,T_4^2\}$ is the free $k$-algebra on $\{S_2^2,T_4^2\}$.

\medskip

Now we prove that the elements considered in the statements of (2), (3), (4) and (5) are symmetric.
Consider first the principal involution, that is, the one in 
Lemma~\ref{lem:equivalentinvolutionHeisenbergalgebra}(3).

$$V^* =  \frac{1}{2}(xy+yx)^* = \frac{1}{2}(xy+yx) = V$$
$$(V-\frac{1}{3}z)^* =  V+\frac{1}{3}z,\qquad (V+\frac{1}{3}z)^*  =  V-\frac{1}{3}z$$
\begin{eqnarray*}
S_1^*&=& \Big(z^{-1}\Big((V-\frac{1}{3}z)(V+\frac{1}{3}z)^{-1}+ (V-\frac{1}{3}z)^{-1}(V+\frac{1}{3}z)\Big) z^{-1}\Big)^*\\
&=& z^{-1}\Big((V-\frac{1}{3}z)^{-1}(V+\frac{1}{3}z)+(V-\frac{1}{3}z)(V+\frac{1}{3}z)^{-1} \Big) z^{-1} \\
&=& S_1 \end{eqnarray*}
\begin{eqnarray*}
((z+y^2)^{-1}(z-y^2))^* & = &  (-z-y^2)(-z+y^2)^{-1} \\
&=& (z+y^2)(z-y^2)^{-1}
\end{eqnarray*}

\begin{eqnarray*}
T_1^*&=& ((z+y^2)^{-1}(z-y^2)S(z+y^2)(z-y^2)^{-1})^* \\
&=&  (z+y^2)^{-1}(z-y^2)S(z+y^2)(z-y^2)^{-1} \\
&=& T_1.
\end{eqnarray*}
Similarly, $S_2^*=S_2$ and $T_3^*=T_3$.

Consider now the involution in Lemma~\ref{lem:equivalentinvolutionHeisenbergalgebra}(2).

$$V^* =  \frac{1}{2}(xy+yx)^* =\frac{1}{2}(xy+yx) =V$$
$$(V-\frac{1}{3}z)^*=  V+\frac{1}{3}z,\qquad (V+\frac{1}{3}z)^*  =  V-\frac{1}{3}z$$
\begin{eqnarray*}
S_1^*&=& \Big(z^{-1}\Big((V-\frac{1}{3}z)(V+\frac{1}{3}z)^{-1}+ (V-\frac{1}{3}z)^{-1}(V+\frac{1}{3}z)\Big) z^{-1}\Big)^*\\
&=& z^{-1}\Big((V-\frac{1}{3}z)^{-1}(V+\frac{1}{3}z)+(V-\frac{1}{3}z)(V+\frac{1}{3}z)^{-1} \Big) z^{-1} \\
&=& S_1 
\end{eqnarray*}
\begin{eqnarray*}((z+y^2)^{-1}(z-y^2))^* & = &  (-z-y^2)(-z+y^2)^{-1} \\
&=& (z+y^2)(z-y^2)^{-1}
\end{eqnarray*}
\begin{eqnarray*}T_1^*&=& ((z+y^2)^{-1}(z-y^2)S(z+y^2)(z-y^2)^{-1})^* \\
&=&  (z+y^2)^{-1}(z-y^2)S(z+y^2)(z-y^2)^{-1} \\
&=& T_1.\end{eqnarray*}
Similarly $S_2^*=S_2$ and $T_3^*=T_3$.

Finally consider the involution in Lemma~\ref{lem:equivalentinvolutionHeisenbergalgebra}(1).
$$V^* =  \frac{1}{2}(xy+yx)^* = -\frac{1}{2}(xy+yx) = -V$$
$$(V-\frac{1}{3}z)^* =  -V-\frac{1}{3}z,\qquad (V+\frac{1}{3}z)^*  = -V+\frac{1}{3}z$$
\begin{eqnarray*}
S_1^*&=& \Big(z^{-1}\Big((V-\frac{1}{3}z)(V+\frac{1}{3}z)^{-1}+ (V-\frac{1}{3}z)^{-1}(V+\frac{1}{3}z)\Big) z^{-1}\Big)^*\\
&=& z^{-1}\Big((-V+\frac{1}{3}z)^{-1}(-V-\frac{1}{3}z)+(-V+\frac{1}{3}z)(-V-\frac{1}{3}z)^{-1} \Big) z^{-1} \\
&=& z^{-1}\Big((V-\frac{1}{3}z)^{-1}(V+\frac{1}{3}z)+ (V-\frac{1}{3}z)(V+\frac{1}{3}z)^{-1}\Big) z^{-1}\\
&=& S_1 
\end{eqnarray*}
$$((z^2+y^3)^{-1}(z^2-y^3))^*  =   (z^2+y^3)(z^2-y^3)^{-1}$$
\begin{eqnarray*}
{T_2}^*&=& ((z^2+y^3)^{-1}(z^2-y^3)S_1(z^2+y^3)(z^2-y^3)^{-1})^* \\
&=&  (z^2+y^3)^{-1}(z^2-y^3)S_1(z^2+y^3)(z^2-y^3)^{-1} \\
&=& T_2.
\end{eqnarray*}
Similarly $S_2^*=S_2$ and $T_4^*=T_4$.
\end{proof}

\begin{theo}\label{theo:symmetricnilpotentLie}
Let $k$ be a field of  characteristic zero and  $L$ be a nonabelian 
nilpotent Lie $k$-algebra.
Let $U(L)$ be the universal enveloping algebra of $L$ and
$\mathfrak{D}(L)$ be the Ore  ring of fractions of $U(H)$.
Then, for any involution $*\colon L\rightarrow L,$ $f\mapsto f^*$,
there exist nonzero symmetric elements $U,V\in L$ such that the
$k$-subalgebra of $\mathfrak{D}(L)$ generated by $\{U,U^{-1},V,V^{-1}\}$
is the free group $k$-algebra on $\{U,V\}$. 
\end{theo}

\begin{proof}
By Proposition~\ref{prop:involutionnilpotent}, there exists a $*$-invariant
Heisenberg Lie $k$-subalgebra $H$ generated by two elements $x,y$ and such that $*$,
when restricted to $H$ is one of the three involutions in 
Lemma~\ref{lem:equivalentinvolutionHeisenbergalgebra}. Since $U(L)$ is an Ore domain,
then $U(H)$ is also an Ore domain. Thus
the division ring generated by $U(H)$ inside $\mathfrak{D}(L)$ is $\mathfrak{D}(H)$.
By Theorem~\ref{theo:freegroupHeisenberg}, there exist elements $U$ and $V$ as desired.
\end{proof}




\section{Free group algebras in division rings generated by 
universal enveloping algebras of residually nilpotent Lie algebras}
\label{sec:residuallynilpotent}

Let $k$ be a field,  $L$ be a  Lie $k$-algebra and $U(L)$ its universal enveloping
algebra. 
It was proved in \cite{Cohnembeddingrings} that $U(L)$ can be embedded in a division ring.  
Two similar proofs of this fact were given in \cite{Lichtmanvaluationmethods} and 
\cite{Huishiskewfields}. Moreover, the division ring constructed in the foregoing three papers is the same
\cite[Theorem~8]{Huishiskewfields}.
We will work with the construction of the skew field $\mathfrak{D}(L)$ which contains $U(L)$
and it is generated by $U(L)$ (as a division ring) given by Lichtman in~\cite{Lichtmanvaluationmethods}. 
The interested reader can also find this construction in \cite[Section~2.6]{Cohnskew}. 
Of course, when $U(L)$ is an Ore domain, $\mathfrak{D}(L)$ is the Ore ring of fractions of $U(L)$. 

In this subsection, we want to obtain free
group algebras in $\mathfrak{D}(L)$, where $L$ is a (generalization
of) a residually nilpotent Lie algebra, from the ones obtained in
$\mathfrak{D}(H)$, where $H$ is the Heisenberg Lie algebra. The technique we will use is
from~\cite[Section~4]{FerreiraGoncalvesSanchez2}. 

For that, we will need some results on the division ring $\mathfrak{D}(L)$. 
For example, $\mathfrak{D}(L)$ is well behaved for Lie subalgebras of $L$ as shown in 
\cite[Proposition~2.5]{Lichtmanuniversalfields}. More precisely, if $N$ is a Lie subalgebra
of $L$, then the natural embedding $U(N)\hookrightarrow U(L)$ can be extended to an embedding
$\mathfrak{D}(N)\hookrightarrow \mathfrak{D}(L)$. Furthermore, 
if $\mathcal{B}_N$ is a basis of $N$ and $\mathcal{C}$
is a set of elements of $L\setminus N$ such that $\mathcal{B}_N\cup
\mathcal{C}$  is a basis of $L$, then the standard monomials in
$\mathcal{C}$
 are linearly independent over
$\mathfrak{D}(N)$. Notice that if $U(L)$ is an Ore domain, then these assertions
are easily verified.

\medskip

Let $k$ be a field and $R$ be a $k$-algebra. Suppose that $\delta\colon R\rightarrow R$
is a $k$-derivation of $R$, that is, $\delta$ is
a $k$-linear map such that $\delta(ab)=\delta(a)b+a\delta(b)$ for all $a,b\in R$. 
We will consider the skew polynomial ring
$R[x;\delta]$. The elements of $R[x;\delta]$ are ``right
polynomials'' of the form $\sum_{i=0}^n x^ia_i$, where the
coeficients $a_i$ are in $R$. The multiplication is determined by
$$ax=xa+\delta(a) \quad \textrm{for all } a\in R.$$

Given $R[x;\delta]$, one can construct the \emph{formal
pseudo-differential operator ring}, denoted $R((t_x;\delta))$,
consisting of the formal Laurent series $\sum_{i=n}^\infty
t_x^ia_i$, with $n\in\mathbb{Z}$ and coefficients $a_i\in R$,
satisfying $at_x^{-1}=t_x^{-1}a+\delta(a)$ for all $a\in R$.
Therefore
\begin{equation}
at_x=t_xa-t_x\delta(a)t_x = \sum_{i=1}^\infty
t_x^i(-1)^{i-1}\delta^{i-1}(a),
\end{equation}
for any $a\in R$.

The subset $R[[t_x;\delta]]$ of $R((t_x;\delta))$ consisting of the
Laurent series of the form $\sum_{i=0}^\infty t_x^ia_i$ is a
$k$-subalgebra of $R((t_x;\delta))$.
The set $\mathcal{S}=\{1,t_x,t_x^2,\dots\}$ is a
left denominator set of $R[[t_x;\delta]]$ such that the Ore
localization $\mathcal{S}^{-1}R[[t_x;\delta]]$ is the $k$-algebra
$R((t_x;\delta))$, see for example \cite[Theorem~2.3.1]{Cohnskew}.
If $R$ is a domain, then a series  $f\in R((t_x;\delta))$ is invertible
if, and only if, the coefficient of the least element in the support of $f$
is invertible in $R$.
Notice that there is a natural embedding $R[x;\delta]\hookrightarrow
R((t_x;\delta_x))$ sending $x$ to $t_x^{-1}$. 

In what follows, $R[y;\delta_y][x;\delta_x]$ means  polynomials of the form
$\sum_{i=0}^n x^i f_i$ where each $f_i\in R[y;\delta_y]$ and $\delta_x$ is a $k$-derivation of $R[y;\delta_y]$.
Also $R((t_y;\delta_y))((t_x;\delta_x))$ is the ring of series of the form
$\sum_{i=n}^\infty
t_x^if_i$, with $n\in\mathbb{Z}$, coefficients $f_i\in R((t_y;\delta_y))$ and 
$\delta_x$ is a $k$-derivation of $R((y;\delta_y))$.

\medskip

Let $k$ be a field.
Let $L$  a Lie $k$-algebra generated by two elements $u,v$. Let
$H=\langle x,y\mid [[y,x],x]=[[y,x],y]=0\rangle$ be the Heisenberg
Lie $k$-algebra. Suppose that there exists a Lie $k$-algebra
homomorphism \begin{equation*}  L\stackrel{\rho}{\rightarrow} H,\
u\mapsto x,\ v\mapsto y.
\end{equation*}
Define $w=[v,u]$ and $z=[y,x]$. Let $N=\ker \rho$. Thus $N$ is a
(Lie) ideal of $L$.

By the universal property of  universal enveloping algebras, $\rho$
can be uniquely extended to a  $k$-algebra homomorphism $\psi\colon
U(L)\rightarrow U(H)$ between the corresponding universal enveloping
algebras. Note that $\ker \psi$ is the ideal of $U(L)$ generated by
$N$. The restriction $\psi_{|U(N)}$ coincides with the augmentation
map $\varepsilon\colon U(N)\rightarrow k$.

By the PBW-Theorem, the elements of $U(H)$ are uniquely expressed as
finite sums  $\sum_{l,m,n\geq 0} x^ly^mz^na_{lmn}$, with
$a_{lmn}\in k$. Let $\delta_x$ be the inner $k$-derivation of $U(H)$
determined by $x$, i.e. $\delta_x(f)=[f,x]=fx-xf$ for all $f\in
U(H)$. It can be proved that
\begin{equation}\label{eq:U(L)aspolynomials}
U(H)=k[z][y][x;\delta_x].
\end{equation}
\begin{equation}\label{eq:U(H)insideseries}
U(H)\hookrightarrow k((t_z))((t_y))((t_x;\delta_x)), \ z\mapsto t_z^{-1},\ y\mapsto t_y^{-1},\ 
x\mapsto t_x^{-1}.
\end{equation}

Consider now  $U(L)$, the universal enveloping algebra of $L$. By
the PBW-Theorem, the elements of $U(L)$ can be uniquely expressed as
finite sums \linebreak $\sum\limits_{l,m,n\geq 0} u^lv^mw^nf_{lmn}$ with $f_{lmn}\in
U(N)$. Since $N$ is an ideal of $L$, the inner derivations
$\delta_u$, $\delta_v$, $\delta_w$ of $U(L)$ defined by $u,v,w$,
respectively, are such that $\delta_u(U(N))\subseteq U(N)$,
$\delta_v(U(N))\subseteq U(N)$, $\delta_w(U(N))\subseteq U(N)$. The
$k$-subalgebra of $U(L)$ generated by $U(N)$ and $w$ is
$U(N)[w;\delta_w]$. Since
$\delta_v(w)\in U(N)\subseteq U(N)[w;\delta_w]$, 
the $k$-subalgebra of $U(L)$ generated by $U(N)$ and
$\{w,v\}$ is $U(N)[w;\delta_w][v;\delta_v]$. Furthermore, since
$\delta_u(v)=w$ and $\delta_u(w)\in U(N)$,
\begin{equation}\label{eq:U(H)aspolynomials}
U(L)=U(N)[w;\delta_w][v;\delta_v][u;\delta_u].
\end{equation}
\begin{eqnarray}\label{eq:U(L)insideseries}
U(L)\hookrightarrow
U(N)((t_w;\delta_w))((t_v;\delta_v))((t_u;\delta_u)), & w\mapsto t_w^{-1},\ v\mapsto t_v^{-1},\ u\mapsto t_u^{-1}, \\
 &f\mapsto \varepsilon(f), \textrm{ for all } f\in U(N). \nonumber
\end{eqnarray}

In this setting, the next two lemmas are \cite[Lemmas~4.1,4.2]{FerreiraGoncalvesSanchez2}.

\begin{lem}\label{lem:commutativediagram}
There exists a commutative diagram of embeddings of $k$-algebras
\begin{equation*}\label{eq:diagramnotOre}
\xymatrixcolsep{0.0001cm}\xymatrix{U(L)=U(N)[w;\delta_w][v;\delta_v][u;\delta_u]\ar@{^{(}->}[rr]\ar@{^{(}->}[dd]
\ar@{^{(}->}[rd] & &
\mathfrak{D}(N)[w;\delta_w][v;\delta_v][u;\delta_u]\ar@{^{(}->}[dd]
\ar@{_{(}->}[ld]\\ &  \mathfrak{D}(L) \ar@{^{(}->}[dr] &\\
U(N)((t_w;\delta_w))((t_v;\delta_v))((t_u;\delta_u))\ar@{^{(}->}[rr]
& & \mathfrak{D}(N)((t_w;\delta_w))((t_v;\delta_v))((t_u;\delta_u))
& } 
\end{equation*}\qed
\end{lem}

\begin{lem} \label{lem:morphismsofseries}
Let $\varepsilon\colon U(N)\rightarrow k$ denote the augmentation
map. The following hold true.
\begin{enumerate}[\normalfont (1)]
\item There exists a  $k$-algebra homomorphism
$$\Phi_w\colon U(N)((t_w;\delta_w))\rightarrow k((t_z)),\quad \sum_{i}t_w^if_i\mapsto \sum_{i}t_z^i\varepsilon(f_i),$$
where $f_i\in U(N)$ for each $i$.
\item There exists a $k$-algebra homomorphism
$$\Phi_v\colon U(N)((t_w;\delta_w))((t_v;\delta_v))
\rightarrow k((t_z))((t_y)),\quad \sum_{i}t_v^ig_i\mapsto
\sum_{i}t_y^i\Phi_w(g_i),$$ where $g_i\in U(N)((t_w;\delta_w))$ for
each $i$.
\item There exists a $k$-algebra homomorphism
\begin{eqnarray*}
\Phi_u\colon U(N)((t_w;\delta_w))((t_v;\delta_v))((t_u;\delta_u))
&\rightarrow& k((t_z))((t_y))((t_x;\delta_x)), \\
\sum_{i}t_u^ih_i &\mapsto& \sum_{i}t_x^i\Phi_v(h_i), \end{eqnarray*}
where $h_i\in U(N)((t_w;\delta_w))((t_y;\delta_y))$ for each $i$ and
extending the embeddings of \eqref{eq:U(H)insideseries} and \eqref{eq:U(L)insideseries}. \qed
\end{enumerate}
\end{lem}

Now we turn our attention to $k$-involutions of $U(L)$ induced from the ones in $L$.

\begin{lem}\label{lem:involutionbehaveswell}
Let $k$ be a field and $L$ be a Lie $k$-algebra generated by two elements $u,v$. Let
$H=\langle x,y\mid [[y,x],x]=[[y,x],y]=0\rangle$ be the Heisenberg
Lie $k$-algebra. Suppose that there exists a Lie $k$-algebra
homomorphism $L\stackrel{\rho}{\rightarrow} H,\
u\mapsto x,\ v\mapsto y.$
Let $N=\ker \rho$. 
Consider the induced $k$-algebra homomorphism $\psi\colon U(L)\rightarrow U(H)$.
Suppose that $*\colon L\rightarrow L$ is an involution in $L$ such that
$N$ is a $*$-invariant ideal of $L$ and call again $*$ the induced involution on $H\cong L/N$.
Then $\psi(f^*)=\psi(f)^*$ for all $f\in U(L)$.
\end{lem}

\begin{proof}
Define $w=[v,u]$ and $z=[y,x]$.
Since $*$ is the induced involution on $H\cong L/N$, then $\psi(u^*)=\psi(u)^*=x^*$,
$\psi(v^*)=\psi(v)^*=y^*$, $\psi(w^*)=\psi(w)^*=z^*$ and
$\psi(f^*)=\psi(f)^*=\varepsilon(f)\in k$ for all $f\in U(N)$.

Given $\sum\limits_{l,m,n\geq 0} u^lv^mw^nf_{lmn}$ with $f_{lmn}\in
U(N)$, we have 
\begin{eqnarray*}\psi\Big(\sum\limits_{l,m,n\geq 0} u^lv^mw^nf_{lmn}\Big)^* & =&
\Big(\sum\limits_{l,m,n\geq 0} x^ly^mz^n\varepsilon(f_{lmn})\Big)^* \\ & =&
\sum\limits_{l,m,n\geq 0} (z^*)^n(y^*)^m(x^*)^l\varepsilon(f_{lmn})
\end{eqnarray*}
On the other hand,
\begin{eqnarray*}
\psi\Big(\Big(\sum\limits_{l,m,n\geq 0} u^lv^mw^n f_{lmn}\Big)^*\Big) & = &
 \psi\Big(\sum\limits_{l,m,n\geq 0} f_{lmn}^*(w^*)^n(v^*)^m
(u^*)^l\Big) \\ & = &\sum\limits_{l,m,n\geq 0} (z^*)^n(y^*)^m(x^*)^l\varepsilon(f_{lmn}),
\end{eqnarray*}
as desired.
\end{proof}

It is well known that any $k$-involution on $L$ can be extended to 
a $k$-involution of $U(L)$. Moreover, 
it was proved in \cite[Proposition~5]{Cimpricfreefieldmany} that any $k$-involution of $L$
can be uniquely extended to a $k$-involution of $\mathfrak{D}(L)$. See also 
\cite[Proposition~2.1]{FerreiraGoncalvesSanchez2}.
With this in mind, we are ready to prove the main result of this section.

\begin{theo}\label{theo:freesymmetricresiduallynilpotent}
Let $k$ be a field of  characteristic zero, let $H=\langle x,y\mid
[[y,x],x]=[[y,x],y]=0\rangle$ be the Heisenberg Lie $k$-algebra and
let $L$  be a Lie $k$-algebra generated by two elements $u,v$.
Suppose that there exists a Lie $k$-algebra homomorphism 
\begin{equation*} L{\rightarrow} H,\ u\mapsto x,\ v\mapsto y,
\end{equation*}
with kernel $N$.
Let $w=[v,u]$, $V=\frac{1}{2}(uv+vu)$, and consider the following
elements of $\mathfrak{D}(L)$:
$$S=(V-\frac{1}{3}w)(V+\frac{1}{3}w)^{-1},\qquad
T=(w+v^2)^{-1}(w-v^2)S(w+v^2)(w-v^2)^{-1},$$ 
$$S_1=w^{-1}\Big((V-\frac{1}{3}w)(V+\frac{1}{3}w)^{-1}+ (V-\frac{1}{3}w)^{-1}(V+\frac{1}{3}w)\Big) w^{-1},$$
$$T_1=(w+v^2)^{-1}(w-v^2)S_1(w+v^2)(w-v^2)^{-1},$$
 $$T_2=(w^2+v^3)^{-1}(w^2-v^3)S_1(w^2+v^3)(w^2-v^3)^{-1}.$$
Then the following hold true.
\begin{enumerate}[\rm(1)]
\item The $k$-subalgebra of
$\mathfrak{D}(L)$
generated by $\{S,S^{-1},T,T^{-1}\}$  is the free group $k$-algebra on the set
$\{S,T\}$.

\item Suppose that $*\colon L\rightarrow L$ is an involution in $L$ such that
$N$ is a $*$-invariant ideal of $L$ and that the induced involution on $H\cong L/N$
is one of the involutions in Lemma~\ref{lem:equivalentinvolutionHeisenbergalgebra}~(2) and (3).
Then the following hold true.
\begin{enumerate}[\rm(a)]
\item The elements $S_1S_1^*$ and $T_1T_1^*$ are symmetric.
\item The $k$-subalgebra of $\mathfrak{D}(L)$ 
generated by $\{S_1S_1^*, T_1T_1^*\}$ is the free 
 $k$-algebra on $\{S_1S_1^*, T_1T_1^*\}$.
\item  The $k$-subalgebra of $\mathfrak{D}(L)$ generated by 
$$\{1+S_1S_1^*, (1+S_1S_1^*)^{-1}, 1+T_1T_1^*,(1+T_1T_1^*)^{-1}\}$$
 is the free group $k$-algebra on the set $\{1+S_1S_1^*, 1+T_1T_1^*\}$.
\end{enumerate}

\item Suppose that $*\colon L\rightarrow L$ is an involution in $L$ such that
$N$ is a $*$-invariant ideal of $L$ and that the induced involution on $H\cong L/N$
is one of the involutions in Lemma~\ref{lem:equivalentinvolutionHeisenbergalgebra}~(1).
Then the following hold true.
\begin{enumerate}[\rm(a)]
\item The elements $S_1S_1^*$ and $T_2T_2^*$ are symmetric.
\item The $k$-subalgebra of $\mathfrak{D}(L)$ 
generated by $\{S_1S_1^*, T_2T_2^*\}$ is the free 
 $k$-algebra on $\{S_1S_1^*, T_2T_2^*\}$.
\item  The $k$-subalgebra of $\mathfrak{D}(L)$ generated by 
$$\{1+S_1S_1^*, (1+S_1S_1^*)^{-1}, 1+T_2T_2^*,(1+T_2T_2^*)^{-1}\}$$
 is the free group $k$-algebra on the set $\{1+S_1S_1^*, 1+T_2T_2^*\}$.
\end{enumerate}
\end{enumerate}
\end{theo}

\begin{proof}
 Define $z=[y,x]\in H$. Consider the embedding $U(H)\hookrightarrow
k((t_z))((t_y))((t_x;\delta_x))$ given in
\eqref{eq:U(H)insideseries}. Since $k((t_z))((t_y))((t_x;\delta_x))$
is a division $k$-algebra and $U(H)$ is an Ore domain, it extends to
an embedding $\mathfrak{D}(H)\hookrightarrow
k((t_z))((t_y))((t_x;\delta_x))$.

Consider  the embedding
$U(L)\hookrightarrow
U(N)((t_w;\delta_w))((t_v;\delta_v))((t_u;\delta_u))$ given in
\eqref{eq:U(L)insideseries}. Let $\Phi_u\colon
U(N)((t_w;\delta_w))((t_v;\delta_v))((t_u;\delta_u)) \rightarrow
k((t_z))((t_y))((t_x;\delta_x))$  be the homomorphism given in
Lemma~\ref{lem:morphismsofseries}.

Define the following elements in $\mathfrak{D}(H)$:
$V_H=\frac{1}{2}(xy+yx)$,
$$S_H=(V_H-\frac{1}{3}z)(V_H+\frac{1}{3}z)^{-1},$$
$$T_H=(z+y^2)^{-1}(z-y^2)S_H(z+y^2)(z-y^2)^{-1}.$$
$${S_1}_H=z^{-1}\Big((V_H-\frac{1}{3}z)(V_H+\frac{1}{3}z)^{-1}+ (V_H-\frac{1}{3}z)^{-1}(V_H+\frac{1}{3}z)\Big) z^{-1},$$
$${T_1}_H=(z+y^2)^{-1}(z-y^2){S_1}_H(z+y^2)(z-y^2)^{-1},$$
$${T_2}_H=(z^2+y^3)^{-1}(z^2-y^3){S_1}_H(z^2+y^3)(z^2-y^3)^{-1},$$

\medskip

\underline{Claim~1}: The elements $V-\frac{1}{3}w$, $V+\frac{1}{3}w$,
$w+v^2$, $w-v^2$, $w^2+v^3$ and $w^2-v^3$ are all invertible in
$U(N)((t_w;\delta_w))((t_v;\delta_v))((t_u;\delta_u))$.

We proceed to prove claim~1. We begin with the element
$w+v^2=t_w^{-1}+t_v^{-2}$. As a series in $t_v$, this element is
invertible in $U(N)((t_w;\delta_w))((t_v;\delta_v))$ if and only if
the coefficient of $t_v^{-2}$ is invertible in the ring of
coefficients  $U(N)((t_w;\delta_w))$. The coefficient is $1$, which
is clearly invertible. Similarly, it can be proved that $w-v^2$, $w^2+v^3$, and $w^2-v^3$
are invertible. Now we show that $V+\frac{1}{3}w$ is invertible in
$U(N)((t_w;\delta_w))((t_v;\delta_v))((t_u;\delta_u))$. First we
obtain an expression of $V+\frac{1}{3}w$ as a series in $t_u$.
\begin{eqnarray}
V+\frac{1}{3}w & = & \frac{1}{2} (uv+vu) +\frac{1}{3}w \nonumber\\ 
& = & \frac{1}{2}(uv+[v,u] +uv) + \frac{1}{3}w \nonumber \\ 
& = & \frac{1}{2}w+  uv + \frac{1}{3}w \nonumber \\ 
& = & \frac{5}{6}w +uv \nonumber\\ 
& = & \frac{5}{6}t_w^{-1} +t_u^{-1}t_v^{-1}  \label{eq:coefseries0}.
\end{eqnarray}
Thus, as a series in $t_u$, the coefficient of the least element in
the support  of $V+\frac{1}{3}w^3$ is $t_v^{-1}$, which is invertible in $U(N)((t_w;\delta_w))((t_v;\delta_v))$.
Hence $V+\frac{1}{3}w^3$ is invertible in
$U(N)((t_w;\delta_w))((t_v;\delta_v))((t_u;\delta_u))$.
The case of
$V-\frac{1}{3}w$ is shown analogously, and the claim is proved.

\medskip

(1) By Claim~1, $S$ and $T$ are invertible in  $U(N)((t_w;\delta_w))((t_v;\delta_v))((t_u;\delta_u))$
and  we  have $\Phi_u(V)=V_H$, $\Phi_u(S)=S_H$ and
$\Phi_u(T)=T_H$. By Theorem~\ref{theo:freegroupHeisenberg},  the $k$-algebra
generated by $\{S_H,\, S_H^{-1},\, T_H,\, T_H^{-1}\}$ is the free group $k$-algebra on the set
$\{S_H,T_H\}$. By Lemma~\ref{lem:commutativediagram},
$V$, $S$ and $T$ belong to $\mathfrak{D}(L)$. Therefore, the
elements $S$ and $T$ are nonzero and invertible in
$\mathfrak{D}(L)$, and the $k$-subalgebra generated by $\{S,\, S^{-1},\, T,\, T^{-1}\}$ is the free group
$k$-algebra on the set $\{S,T\}$.

(2) (a) It is clear.

(b) We will prove in  detail the result for the involution in 
Lemma~\ref{lem:equivalentinvolutionHeisenbergalgebra}(2), the
other case can be shown similarly. Let $n_u,n_v,n_w,\in N$ be such that $$u^*=u+n_u,\ v^*=v+n_v,\ w^*=-w+n_w$$

\underline{Claim~2:} The elements $$(V+\frac{1}{3}w)^*,\ (V-\frac{1}{3}w)^*,\ (w+v^2)^*,\ (w-v^2)^*$$
 belong to and are invertible in $U(N)((t_w;\delta_w))((t_v;\delta_v))((t_u;\delta_u))$. 

From Claim~2, it follows that the elements $S_1^*, T_1^*\in U(N)((t_w;\delta_w))((t_v;\delta_v))((t_u;\delta_u))$.
By Lemma~\ref{lem:involutionbehaveswell}, $\Phi_u(Z^*)=\Phi_u(Z)^*$ where $Z$ is any of the elements in Claim~2.
Thus, by Theorem~\ref{theo:freegroupHeisenberg}(2)(a), 
\begin{equation}\label{eq:imagephiu}
\Phi_u(S_1^*)=\Phi_u(S_1)^*=S_{1H}^*=S_{1H},\quad \Phi_u(T_1^*)=\Phi_u(T_1)^*=T_{1H}^*=T_{1H}.
\end{equation}
Hence $\Phi_u(S_1S_1^*)=S_{1H}^2$ and $\Phi_u(T_1T_1^*)=T_{1H}^*$. 
By Theorem~\ref{theo:freegroupHeisenberg}(2)(d), the $k$-algebra
generated by $\{S_{1H}^2, T_{1H}^2\}$ is the free algebra on $\{S_{1H}^2,T_{1H}^2\}$. Therefore the result follows.

We proceed to prove Claim~2.

\begin{align}\label{eq:coefseries1}
(V+\frac{1}{3}w)^* & =  \left(\frac{1}{2}(uv+vu)+\frac{1}{3}w \right)^* \nonumber \\
& =  \frac{1}{2}((u+n_u)(v+n_v)+(v+n_v)(u+n_u)) +\frac{1}{3}(-w+n_w)  \nonumber \\
& =  \frac{1}{2}(uv+vu+un_v+n_vu+n_uv+vn_u+n_un_v+n_vn_u) \nonumber \\
& \qquad \quad -\frac{1}{3}w+ \frac{1}{3}n_w  \nonumber \\
&=\frac{1}{2}(uv+uv+[v,u] +un_v+un_v+[n_v,u]+vn_u+vn_u \nonumber\\
&\qquad\quad+[n_u,v]+n_un_v+n_vn_u)-\frac{1}{3}w+ \frac{1}{3}n_w \nonumber  \\ 
& =  u(v+n_v)+ vn_u+\frac{1}{6}w+f_1 \nonumber\\
& =  t_u^{-1}(t_v^{-1}+n_v)+t_v^{-1}n_u+\frac{1}{6}t_w^{-1}+f_1, 
\end{align}
where $f_1\in U(N)$.

\begin{align}
(V-\frac{1}{3}w)^* & =  \left(\frac{1}{2}(uv+vu)-\frac{1}{3}w \right)^* \nonumber\\
& =  \frac{1}{2}((u+n_u)(v+n_v)+(v+n_v)(u+n_u)) -\frac{1}{3}(-w+n_w)  \nonumber\\
& =  \frac{1}{2}(uv+vu+un_v+n_vu+n_uv+vn_u+n_un_v+n_vn_u) \nonumber\\
& \qquad \quad +\frac{1}{3}w- \frac{1}{3}n_w  \nonumber\\
&=\frac{1}{2}(uv+uv+[v,u] +un_v+un_v+[n_v,u]+vn_u+vn_u\nonumber\\ 
&\qquad \quad +[n_u,v]+n_un_v+n_vn_u)+\frac{1}{3}w- \frac{1}{3}n_w  \nonumber\\ 
& =  u(v+n_v)+ vn_u+\frac{5}{6}w+f_2 \nonumber\\
& =  t_u^{-1}(t_v^{-1}+n_v)+t_v^{-1}n_u+\frac{5}{6}t_w^{-1}+f_2, \label{eq:coefseries2}
\end{align}
where $f_2\in U(N)$. Note that the element $(t_v^{-1}+n_v)$ is invertible in $U(N)((t_w;\delta_w))((t_v;\delta_v))$.
Thus $(V+\frac{1}{3})^*$ and $(V-\frac{1}{3})^*$ are invertible in $U(N)((t_w;\delta_w))((t_v;\delta_v))((t_u;\delta_u))$.

There exist $f_3,f_4\in U(N)$ such that

\begin{eqnarray}
(w+v^2)^*&=& -w+n_w+(v+n_v)^2 \nonumber\\
& = & v^2+vn_v+n_vv+n_v^2-w+n_w \nonumber\\
& = & v^2+2vn_v-w+[n_v,v]+ n_v^2+n_w \nonumber\\
& = & t_v^{-2} +2t_v^{-1}n_v -t_w^{-1} + f_3 \label{eq:coefseries3}
\end{eqnarray}

\begin{eqnarray}
(w-v^2)^*&=& -w+n_w-(v+n_v)^2 \nonumber\\
& = & -v^2-vn_v-n_vv-n_v^2-w+n_w \nonumber\\
& = & -v^2-2vn_v-w-[n_v,v]- n_v^2+n_w \nonumber\\
& = & -t_v^{-2} -2t_v^{-1}n_v -t_w^{-1} + f_4 \label{eq:coefseries4}
\end{eqnarray}

The elements $(w+v^2)^*,(w-v^2)^*$ are invertible because the coefficient of
$t_v^{-2}$ is $\pm1$, which is clearly invertible. And the claim is proved.

\bigskip

(c) By Theorem~\ref{theo:freegroupHeisenberg}(2)(d),  the $k$-subalgebra
generated by $\{1+S_{1H}^2, (1+S_{1H})^{-1}, 1+T_{1H}^2, (1+T_{1H})^{-1}\}$ is the free group $k$-algebra
on the set $\{1+S_{1H}^2, 1+T_{1H}^2\}$. Moreover, by \eqref{eq:imagephiu}, $\Phi_u(1+S_1S_1^*)=1+S_{1H}^2$ and
$\Phi_u(1+T_1T_1^*)=1+T_{1H}^2$. Therefore it is enough to prove that the elements
$1+S_1S_1^*$ and $1+T_1T_1^*$ are invertible in $U(N)((t_w;\delta_w))((t_v;\delta_v))((t_u;\delta_u))$.

By \eqref{eq:coefseries0}, $V-\frac{1}{3}w$ and $V+\frac{1}{3}w$ are series of the form
$t_u^{-1}t_v^{-1}(1+h_1)$ where $h_1$ is a series on positive powers of $t_u$ with
coefficients in $U(N)((t_w;\delta_w))((t_v;\delta_v))$. Hence
$(V-\frac{1}{3}w)^{-1}$ and $(V+\frac{1}{3}w)^{-1}$ are series of the form
$t_vt_u(1+h_2)$ where $h_2$ is a series on positive powers of $t_u$ with
coefficients in $U(N)((t_w;\delta_w))((t_v;\delta_v))$. Using that $w^{-1}=t_w$,
we obtain that $S_1$ is a series of the form $2t_w^2+h_3$ where $h_3$
is a series on positive powers of $t_u$ with
coefficients in $U(N)((t_w;\delta_w))((t_v;\delta_v))$. 

By \eqref{eq:coefseries1}, \eqref{eq:coefseries2}, $(V-\frac{1}{3}w)^*$ and 
$(V+\frac{1}{3}w)^*$ are series of the form $t_u^{-1}(t_v^{-1}+n_v)(1+h_4)$, 
where $h_4$ is a series on positive powers of $t_u$ with
coefficients in $U(N)((t_w;\delta_w))((t_v;\delta_v))$. 
Hence
$((V-\frac{1}{3}w)^*)^{-1}$ and $((V+\frac{1}{3}w)^*)^{-1}$ are series of the form
$(t_v^{-1}+n_v)^{-1}t_u(1+h_5)$ where $h_5$ is a series on positive powers of $t_u$ with
coefficients in $U(N)((t_w;\delta_w))((t_v;\delta_v))$. Using that $(w^*)^{-1}=(-t_w^{-1}+n_w)^{-1}$,
we obtain that $S_1^*$ is a series of the form $2t_w^2+h_6$ where $h_6$
is a series on positive powers of $t_u$ with
coefficients in $U(N)((t_w;\delta_w))((t_v;\delta_v))$. From these considerations
it follows that $1+S_1S_1^*$ is a series of the form $1+4t_w^4+h_7$  where $h_7$
is a series on positive powers of $t_u$ with
coefficients in $U(N)((t_w;\delta_w))((t_v;\delta_v))$. Now $1+4t_w^4$ is
invertible in $U(N)((t_w;\delta_w))((t_v;\delta_v))$, and $1+S_1S_1^*$ is
therefore invertible in $U(N)((t_w;\delta_w))((t_v;\delta_v))((t_u;\delta_u))$.

Clearly $(w+v^2)$ and $(w-v^2)$ are series of the form $\pm t_v^{-2}(1+g_1)$
where $g_1$ is a series on positive powers of $t_v$ and coefficients in
$U(N)((t_w;\delta_w))$. Thus $(w+v^2)^{-1}$ and $(w-v^2)^{-1}$ are
series of the form  $\pm t_v^{2}(1+g_2)$ 
where $g_2$ is a series on positive powers of $t_v$ and coefficients in
$U(N)((t_w;\delta_w))$. Using that $S_1$ is a series of the form $2t_w^2+h_3$, where $h_3$
is as stated above, we obtain that  $T_1$ is a series  of the form
$2t_w^2+g_3+h_8$ where  $g_3$ is a series on positive powers of $t_v$ and coefficients in
$U(N)((t_w;\delta_w))$ and $h_8$ is a series on positive powers of $t_u$ with
coefficients in $U(N)((t_w;\delta_w))((t_v;\delta_v))$.

By \eqref{eq:coefseries3}, \eqref{eq:coefseries4}, 
$(w+v^2)^*$ and $(w-v^2)^*$ are series of the form $\pm t_v^{-2}(1+g_4)$ 
where $g_4$ is a series on positive powers of $t_v$ and coefficients in
$U(N)((t_w;\delta_w))$. Thus $((w+v^2)^*)^{-1}$ and $((w-v^2)^*)^{-1}$ are
series of the form  $\pm t_v^{2}(1+g_5)$ 
where $g_5$ is a series on positive powers of $t_v$ and coefficients in
$U(N)((t_w;\delta_w))$. Using that $S_1^*$ is a series of the form $2t_w^2+h_6$, where $h_6$
is as stated above, we obtain that  $T_1^*$ is a series  of the form
$2t_w^2+g_6+h_9$ where  $g_6$ is a series on positive powers of $t_v$ and coefficients in
$U(N)((t_w;\delta_w))$ and $h_9$ is a series on positive powers of $t_u$ with
coefficients in $U(N)((t_w;\delta_w))((t_v;\delta_v))$. Therefore
$1+T_1T_1^*$ is a series of the form $1+4t_w^4+g_7+h_{10}$ where 
$g_7$ is a series on positive powers of $t_v$ and coefficients in
$U(N)((t_w;\delta_w))$ and $h_{10}$ is a series on positive powers of $t_u$ with
coefficients in $U(N)((t_w;\delta_w))((t_v;\delta_v))$. Now 
$1+T_1T_1^*$ is invertible because the series $1+4t_w^4+g_7$ is invertible
in $U(N)((t_w;\delta_w))((t_v;\delta_v))$ since $1+4t_w^4$ is invertible in
  $U(N)((t_w;\delta_w))$.

\bigskip

\noindent(3) Suppose that the induced involution on $L/N$ is the one on 
Lemma~\ref{lem:equivalentinvolutionHeisenbergalgebra}(3).

The result follows very much like (2) from the following claim which can be shown as Claim~2.

\underline{Claim~3:} The elements 
$$(V+\frac{1}{3}w)^*,\ (V-\frac{1}{3}w)^*,\ (w^2+v^3)^*,\ (w^3-v^3)^*$$
belong to and are invertible in $U(N)((t_w;\delta_w))((t_v;\delta_v))((t_u;\delta_u))$.

\begin{align*}
(w^2+v^3)^*&= (-w+n_w)^2+(-v+n_v)^3 \\
& =  w^2 - wn_w-n_ww +n_w^2- v^3+v^2n_v+vn_vv +n_vv^2\\ 
& \qquad \quad -vn_v^2-n_v^2v-n_vvn_v+n_v^3\\
& =  -v^3+3v^2n_v-v(3n_v^2+[n_v,v])-[n_v^2,v]+[n_v,v^2]\\ 
& \qquad \quad -[n_v,v]n_v+n_v^3+w^2-2wn_w+[n_w,w]+n_w^2\\
& =  -t_v^{-3}+3t_v^{-2}n_v- t_v^{-1}(3n_v^2+[n_v,v])+t_w^{-2}-2t_w^{-1} +f_5
\end{align*}

\begin{align*}
(w^2-v^3)^*&= (-w+n_w)^2-(-v+n_v)^3 \\
& =   w^2 - wn_w-n_ww +n_w^2-(- v^3+v^2n_v+vn_vv +n_vv^2-vn_v^2\\
& \qquad\qquad\qquad\qquad -n_v^2v-n_vvn_v+n_v^3 )\\
& =  t_v^{-3}-3t_v^{-2}n_v+ t_v^{-1}(3n_v^2+[n_v,v])+t_w^{-2}-2t_w^{-1} +f_6
\end{align*} 

The elements $(w^2+v^3)^*,(w^2-v^3)^*$ are invertible because the coefficient of $t_v^{-3}$
is $\pm1$, which is clearly invertible.
\end{proof}

\begin{coro}\label{coro:freesymmetricresiduallynilpotent}
Let $k$ be a field of  characteristic zero and $K$ be a 
residually nilpotent Lie
$k$-algebra. 
Let $u,v\in K$ be such that $[v,u]\neq 0$ and denote by $L$ the Lie
$k$-subalgebra of $K$ generated by $\{u,v\}$. 

Let $w=[v,u]$, $V=\frac{1}{2}(uv+vu)$, and consider the following
elements of $\mathfrak{D}(L)$:
$$S=(V-\frac{1}{3}w)(V+\frac{1}{3}w)^{-1},\qquad
T=(w+v^2)^{-1}(w-v^2)S(w+v^2)(w-v^2)^{-1},$$ 
$$S_1=w^{-1}\Big((V-\frac{1}{3}w)(V+\frac{1}{3}w)^{-1}+ (V-\frac{1}{3}w)^{-1}(V+\frac{1}{3}w)\Big) w^{-1},$$
$$T_1=(w+v^2)^{-1}(w-v^2)S_1(w+v^2)(w-v^2)^{-1},$$
 $$T_2=(w^2+v^3)^{-1}(w^2-v^3)S_1(w^2+v^3)(w^2-v^3)^{-1}.$$
Then the following hold true.

\begin{enumerate}[\rm(1)]
	\item The Lie $k$-algebra $L/[[L,L],L]$ is isomorphic to $H$, the Heisenberg Lie
	$k$-algebra.
	\item The $k$-subalgebra of $\mathfrak{D}(L)$ generated by $\{S,\, S^{-1},\, T,\, T^{-1}\}$ is the free group
$k$-algebra on the set $\{S,T\}$.
	\item Suppose that $L$ is invariant under $*$ and that the induced involution on
	$L/[[L,L],L]$ is one of the involution on Lemma~\ref{lem:equivalentinvolutionHeisenbergalgebra}.
	\begin{enumerate}[\rm(i)]
\item If 	the induced involution on $L/[[L,L],L]$
is one of the involutions in Lemma~\ref{lem:equivalentinvolutionHeisenbergalgebra}~(2) and (3).
Then the following hold true.
\begin{enumerate}[\rm(a)]
\item The elements $S_1S_1^*$ and $T_1T_1^*$ are symmetric.
\item The $k$-subalgebra of $\mathfrak{D}(L)$ 
generated by $\{S_1S_1^*, T_1T_1^*\}$ is the free 
 $k$-algebra on $\{S_1S_1^*, T_1T_1^*\}$.
\item  The $k$-subalgebra of $\mathfrak{D}(L)$ generated by 
$$\{1+S_1S_1^*, (1+S_1S_1^*)^{-1}, 1+T_1T_1^*,(1+T_1T_1^*)^{-1}\}$$
 is the free group $k$-algebra on the set $\{1+S_1S_1^*, 1+T_1T_1^*\}$.
\end{enumerate}

\item If the induced involution on $L/[[L,L],L]$
is one of the involutions in Lemma~\ref{lem:equivalentinvolutionHeisenbergalgebra}~(1).
Then the following hold true.
\begin{enumerate}[\rm(a)]
\item The elements $S_1S_1^*$ and $T_2T_2^*$ are symmetric.
\item The $k$-subalgebra of $\mathfrak{D}(L)$ 
generated by $\{S_1S_1^*, T_2T_2^*\}$ is the free 
 $k$-algebra on $\{S_1S_1^*, T_2T_2^*\}$.
\item  The $k$-subalgebra of $\mathfrak{D}(L)$ generated by 
$$\{1+S_1S_1^*, (1+S_1S_1^*)^{-1}, 1+T_2T_2^*,(1+T_2T_2^*)^{-1}\}$$
 is the free group $k$-algebra on the set $\{1+S_1S_1^*, 1+T_2T_2^*\}$.
	\end{enumerate}
	
\end{enumerate} 
\end{enumerate}

\end{coro}

\begin{proof}
Define $N=[[L,L],L]$. 
Since $L$ is residually nilpotent and not abelian, $[v,u]\in [L,L]\setminus N$.
Thus $L/N$ is not abelian.
Moreover
$L/N$ is a noncommutative 3-dimensional Lie $k$-algebra with basis
$\{\bar{u},\bar{v},\bar{w}\}$,  the classes of $u$, $v$ and $w$ in
$L/N$. Moreover $[L/N,L/N]=k\bar{w}$ which is contained in the
center of $L/N$. Therefore $L/N$ is the Heisenberg Lie $k$-algebra.

By Theorem~\ref{theo:freesymmetricresiduallynilpotent}, the result
holds for $\mathfrak{D}(L)$. Since
$\mathfrak{D}(L)\hookrightarrow\mathfrak{D}(K)$, the result
follows.
\end{proof}

\begin{coro}\label{coro:symmetricresiduallynilpotent}
Let $k$ be a field of characteristic zero and $L$ be a nonabelian residually nilpotent
Lie $k$-algebra endowed with
an involution $*\colon L\rightarrow L$. 
Then there exist symmetric elements $A,B\in \mathfrak{D}(L)$ such that
the $k$-subalgebra of $\mathfrak{D}(L)$ generated by $\{A,A^{-1},B,B^{-1}\}$ is the free
group $k$-algebra on $\{A,B\}$.
\end{coro}
\begin{proof}
Let $N$ be the $*$-invariant ideal $[L,[L,L]]$.  
The Lie $k$-algebra $L/N$ is nilpotent but not abelian and $*$ induces an involution
on $L/N$. By Proposition~\ref{prop:involutionnilpotent}, there exists an invariant
Heisenberg Lie $k$-subalgebra $H$ of $L/N$ such that the restriction of the involution is
one of involutions of Lemma~\ref{lem:equivalentinvolutionHeisenbergalgebra}. Let
$a+N$, $b+N$ be the generators of $H$. Let $M$ be the Lie $k$-subalgebra of $L$ generated
by $N\cup\{a,b\}$. Then $*$ induces an involution $*\colon M\rightarrow M$ by restriction,
$M/N\cong H$ and the induced involution on $M/N$ is one of the involutions of Lemma~\ref{lem:equivalentinvolutionHeisenbergalgebra}. Apply Theorem~\ref{theo:freesymmetricresiduallynilpotent}(2) and (3)
to obtain that $\mathfrak{D}(M)$ satisfies the 
desiered result. Now observe that $\mathfrak{D}(M)\subseteq \mathfrak{D}(L)$.
\end{proof}

\section{Free group algebras in the Ore ring of fractions of universal enveloping
algebras which are Ore domains}\label{sec:Ore}

The main results in this section are 
Theorems~\ref{theo:freegroupOre},\ref{theo:symmetricOre1},\ref{theo:symmetricOre2}. They
all have a similar but technical proof. Thanks to
the results in Section~\ref{sec:freegroupalgebrasdivision}, 
the method can be seen as an improvement of the technique
originally used in the proof of \cite[Theorem~2]{Lichtmanfreeuniversalenveloping} and that
was also used to show \cite[Theorem~5.2]{FerreiraGoncalvesSanchez2}.

\subsection{On conjecture (GA)}

\begin{theo}\label{theo:freegroupOre}
Let $k$ be a field of  characteristic zero and $L$ be a Lie
$k$-algebra whose universal enveloping algebra $U(L)$ is an
Ore domain. Let $u,v\in L$ such that the Lie subalgebra generated by
them is of dimension at least three.

Define $w=[v,u]$, $V=\frac{1}{2}(uv+vu)$, and consider the
following elements of $\mathfrak{D}(L)$ the Ore ring of
fractions of $U(L)$:
$$S=(V-\frac{1}{3}w)(V+\frac{1}{3}w)^{-1},\qquad T=(w+v^2)^{-1}(w-v^2)S(w+v^2)(w-v^2)^{-1}.$$ Then the $k$-subalgebra
of $\mathfrak{D}(L)$ generated by $\{S,\, S^{-1},\, T,\, T^{-1}\}$ is the free group
$k$-algebra on $\{S,T\}$.
\end{theo}

\begin{proof}
Let $L_1$ be the Lie $k$-subalgebra of $L$ generated by $u$ and $v$.
Since $U(L)$ is an Ore domain, $U(L_1)$ is also an Ore domain and
$\mathfrak{D}(L_1)\subseteq \mathfrak{D}(L)$. Thus, we may
suppose that $L$ is generated by $u$ and $v$.

Consider the filtration $F_{\mathbb{Z}}L=\{F_nL\}_{n\in\mathbb{Z}}$ of $L$
given in
Example~\ref{ex:usualfiltrationLiealgebra}. It induces a filtration
$F_\mathbb{Z}U(L)=\{F_nU(L)\}_{n\in\mathbb{Z}}$  on $U(L)$ as shown
in Section~\ref{sec:filtrationuniversal}. Moreover, by
Lemma~\ref{lem:filtrationuniversalenveloping}(1), there exists
an isomorphism of $\mathbb{Z}$-graded $k$-algebras
\begin{equation}\label{eq:isomorphismofgraded}
U(\gr_{F_\mathbb{Z}}(L)) \cong  \gr_{F_\mathbb{Z}}(U(L)),
\end{equation}
which induces a valuation $\upsilon\colon U(L)\rightarrow
\mathbb{Z}\cup\{\infty\}$ as in Section~\ref{sec:generalfiltrations}.
 It can be
extended to a valuation $\upsilon\colon \mathfrak{D}(L)\rightarrow
\mathbb{Z}\cup\{\infty\}$ \cite[Proposition~9.1.1]{Cohnskew}. 
We recall that the
filtration it induces is $F_\mathbb{Z}\mathfrak{D}(L)=\{F_n\mathfrak{D}(L)\}_{n\in\mathbb{Z}}$
where $F_n\mathfrak{D}(L)=\{f\in \mathfrak{D}(L)\colon \upsilon(f)\geq n\}$.

In what follows,  the two
objects in \eqref{eq:isomorphismofgraded} will be identified.
Consider $u,v$ and $w=[v,u]$. Note that $\upsilon(u)=\upsilon(v)=-1$
and $\upsilon(w)=-2$ because $L$ is not two-dimensional.  Denote by
$\bar{u}$, $\bar{v}$ the class of $u,v\in U(L)_{-1}$ and also the
class of $u$ and $v$ in $L_{-1}$. Denote by $\bar{w}$ the class of
$w$ in $U(L)_{-2}$ and in $L_{-2}$.  By
Lemma~\ref{lem:gradedOre}(4), $U(\gr_{F_\mathbb{Z}}(L))$ is an Ore domain.
Let $\mathfrak{D}(\gr_{F_\mathbb{Z}}(L))$ be its Ore ring of fractions.

Now, $\gr_{F_\mathbb{Z}}(L)$ is a (negatively) graded Lie $k$-algebra
which is not abelian $(w\in L_{-2}\setminus L_{-1})$. Thus
$\gr_{F_\mathbb{Z}}(L)$ is a nonabelian residually nilpotent Lie $k$-algebra.
Observe that $[\bar{v},\bar{u}]=\bar{w}$ as elements of
$\gr_{F_\mathbb{Z}}(L)$.

Now define
$\overline{V}=\frac{1}{2}(\bar{u}\bar{v}+\bar{v}\bar{u})$,
$$\overline{S}=(\overline{V}-\frac{1}{3}\bar{w})(\overline{V}+\frac{1}{3}\bar{w})^{-1},\qquad
\overline{T}=(\bar{w}+\bar{v}^2)^{-1}(\bar{w}-\bar{v}^2)\overline{S}(\bar{w}+\bar{v}^2)(\bar{w}-\bar{v}^2)^{-1}.$$
Then Corollary~\ref{coro:freesymmetricresiduallynilpotent}(2) shows
that the $k$-subalgebra of $\mathfrak{D}(\gr_{F_\mathbb{Z}}(L))$ generated
by
$\{\overline{S},{\overline{S}^{\phantom{.}}}^{-1},\overline{T},{\overline{T}^{\phantom{.}}}^{-1}\}$
is the free group $k$-algebra on $\{\overline{S},\overline{T}\}$.
Let $\mathcal{H}$ be the set of homogeneous elements of $\gr_{F_\mathbb{Z}}(U(L))$. From
\eqref{eq:isomorphismofgraded}, and Lemma~\ref{lem:gradedOre} we
obtain the following commutative diagram 
$$\xymatrix{\gr_{F_\mathbb{Z}}(U(L))\cong U(\gr_{F_\mathbb{Z}}(L))\ar@{^{(}->}[r]\ar@{^{(}->}[d]
 &
\mathfrak{D}(\gr_{F_\mathbb{Z}}(L))\\
\mathcal{H}^{-1}\gr_{F_\mathbb{Z}}(U(L))\cong\gr_{F_\mathbb{Z}}(\mathfrak{D}(L))
\ar@{^{(}->}[ur] & },$$ where the diagonal arrow is obtained from the
universal property of the Ore localization. Note that
$\overline{V},\, \overline{V}-\frac{1}{3}\bar{w},\,
\overline{V}+\frac{1}{3}\bar{w},\, \bar{w}+\bar{v}^2,\,
\bar{w}-\bar{v}^2$ are homogeneous elements of degree $-2$ in
$\gr_{F_\mathbb{Z}}(U(L))$. Thus $\overline{S}$,
${\overline{S}^{\phantom{.}}}^{-1}$, $\overline{T}$,
${\overline{T}^{\phantom{.}}}^{-1}$ are in fact homogeneous elements
of degree zero in $\gr_{F_\mathbb{Z}}(\mathfrak{D}(L))$.

Now observe that $S$ and $T$ are elements of $\mathfrak{D}(L)$ such
that $\upsilon(S)=\upsilon(T)=0$ and
$\overline{S}=S+\mathfrak{D}(L)_{>0},\,
\overline{T}=T+\mathfrak{D}(L)_{>0}$ in
$\gr_{F_\mathbb{Z}}(\mathfrak{D}(L))$. By
Proposition~\ref{prop:freeobjecthomogeneous}, the $k$-subalgebra of
$\mathfrak{D}(L)$ generated by $\{S,\, S^{-1},\, T,\, T^{-1}\}$ is
the free group $k$-algebra on $\{S,T\}$.
\end{proof}

When the Lie subalgebra generated by $u$ and $v$ is of dimension
two, we cannot apply the methods developed thus far, but we have the
following consequence of Cauchon's Theorem.

\begin{prop}\label{prop:twodimensionalcase}
Let $k$ be a field of characteristic zero. Let $M$ be the
nonabelian two dimensional Lie $k$-algebra. Thus $M$ has a basis
$\{e,f\}$ such that $[e,f]=f$. Define
$s=(e-\frac{1}{3})(e+\frac{1}{3})^{-1}$ and $u=(1-f)(1+f)^{-1}$.
Consider the embedding $U(M)\hookrightarrow\mathfrak{D}(M)$. Then
 the $k$-algebra generated by the set
$\{S=s,\, S^{-1}, T=usu^{-1},\, T^{-1}\}$ is
the free group $k$-algebra on $\{S,\ T\}$.
\end{prop}

\begin{proof}
Since $[e,f]=ef-fe=f$, $ef=f(e+1)$. Thus  $U(M)$ can be seen as a
skew polynomial $k$-algebra, $U(M)=k[e][f;\sigma]$, where
$\sigma(e)=e+1$.

According to Cauchon's Theorem, if we define
$s=(e-\frac{1}{3})(e+\frac{1}{3})^{-1}$ and $u=(1-f)(1+f)^{-1}$, the $k$-subalgebra
generated by  $\{s,\, s^{-1},\, usu^{-1},\, us^{-1}u^{-1}\}$ is the free group $k$-algebra
on $\{s,\, usu^{-1}\}$.
\end{proof}

Combining together Theorem~\ref{theo:freegroupOre} and
Proposition~\ref{prop:twodimensionalcase}, we obtain the following
result which is \cite[Theorem~4]{Lichtmanfreeuniversalenveloping}.

\begin{theo}\label{theo:freesymmetricOregeneral}
Let $k$ be a field of characteristic zero. Let $L$ be a
noncommutative Lie \mbox{$k$-algebra} such that $U(L)$ is an Ore
domain. Then there exist elements $S,T\in\mathfrak{D}(L)$ such that
the $k$-subalgebra of $\mathfrak{D}(L)$ generated by $\{S,\, S^{-1},\, T,\, T^{-1}\}$ is the
free group  $k$-algebra on $\{S,T\}$.
More precisely, let $u,v\in L$ such that $[v,u]\neq 0$. Then
\begin{enumerate}[\normalfont(1)]
\item if the Lie $k$-subalgebra of $L$ generated by $\{u,v\}$ is of dimension greater than two, then
one can choose $S$ and $T$ as defined in
Theorem~\ref{theo:freegroupOre},
\item if the Lie $k$-subalgebra of $L$ generated by $\{u,v\}$ is of dimension exactly two, then
one can choose $S$ and $T$ as defined in
Proposition~\ref{prop:twodimensionalcase}. \qed
\end{enumerate}
\end{theo}

\subsection{On involutional versions of conjecture (GA)}

Now we turn our attention involutions and the existence of free group
algebras generated by symmetric elements.

\begin{theo}\label{theo:symmetricOre1}
Let $k$ be a field of characteristic zero and $L$ be a Lie $k$-algebra such that
$U(L)$ is an Ore domain. Let $*\colon L\rightarrow L$ be a $k$-involution.
Suppose that there exists an element $a\in L$ such that $[a^*,a]\neq 0$
and the Lie $k$-subalgebra generated by $\{a,a^*\}$ is of dimension at least $3$.

Define $u=a+a^*$, $v=a^*-a$, $w=[v,u]$ and $V=\frac{1}{2}(uv+vu)$, and consider the following
elements of $\mathfrak{D}(L)$:
$$S_1=w^{-1}\Big((V-\frac{1}{3}w)(V+\frac{1}{3}w)^{-1}+ (V-\frac{1}{3}w)^{-1}(V+\frac{1}{3}w)\Big) w^{-1},$$
 $$T_2=(w^2+v^3)^{-1}(w^2-v^3)S_1(w^2+v^3)(w^2-v^3)^{-1}.$$
Then the $k$-subalgebra of $\mathfrak{D}(L)$ generated by 
$$\{1+S_1S_1^*, (1+S_1S_1^*)^{-1}, 1+T_2T_2^*,(1+T_2T_2^*)^{-1}\}$$
 is the free group $k$-algebra on the set $\{1+S_1S_1^*, 1+T_2T_2^*\}$.
\end{theo}

\begin{proof}
Let $L_1$ be the Lie $k$-subalgebra of $L$ generated by $u$ and $v$.

Since $U(L)$ is an Ore domain, $U(L_1)$ is also an Ore domain.
Moreover, $\mathfrak{D}(L_1)\subseteq \mathfrak{D}(L)$. Thus, we may
suppose that $L$ is generated by $u$ and $v$.

Consider the filtration $F_{\mathbb{Z}}L=\{F_nL\}_{n\in\mathbb{Z}}$ of $L$
defined by 
$F_rL=0$ for all $r\geq 0$,  $F_{-1}L=ku$, $F_{-2}L=kv+F_{-1}L$, $F_{-3}L=k[v,u]+F_{-2}L$
and for $n\leq -3$,
$$F_{n-1}L=\sum_{n_1+n_2+\dotsb+n_r\geq (n-1)}[F_{n_1}L,[F_{n_2}L,\dotsc] \dotsb].$$
Observe that, for each $n\in\mathbb{Z}$, there exists   $\mathcal{B}_{n}\subseteq L$
whose classes give a basis of $L_{n}=F_nL/F_{n+1}L$ such that
$\bigcup\limits_{n\in\integers} \mathcal{B}_{n}$ is a basis of $L$.
This filtration on $L$ induces a filtration 
$F_\mathbb{Z}U(L)=\{F_nU(L)\}_{n\in\mathbb{Z}}$  on $U(L)$ as shown
in Section~\ref{sec:filtrationuniversal}. Moreover, by
Lemma~\ref{lem:filtrationuniversalenveloping}, there exists an isomorphism of $\mathbb{Z}$-graded
$k$-algebras
\begin{equation}\label{eq:isomorphismofgraded2}
U(\gr_{F_\mathbb{Z}}(L)) \cong  \gr_{F_\mathbb{Z}}(U(L)),
\end{equation}
which induces a valuation $\upsilon\colon U(L)\rightarrow
\mathbb{Z}\cup\{\infty\}$ as in Section~\ref{sec:generalfiltrations}.
In what follows,  the two
objects in \eqref{eq:isomorphismofgraded2} will be identified
via the isomorphism given in either \cite[Proposition~1]{Vergne} or 
\cite[Lemma~2.1.2]{Boiscorpsenveloppants}. This isomorphism sends
the class of an element of $\mathcal{B}_n$ in $L_n$ to
its class in $U(L)_n$.

Note that each $F_nL$ is invariant under $*$ because $u^*=u$ and $v^*=-v$.
Hence $*$ induces an involution on $\gr_{F_\mathbb{Z}}(L)$ and hence
on $U(\gr_{F_\mathbb{Z}}(L))$. Moreover each $F_nU(L)$ is invariant under $*$,
and thus $*$ also induces an involution on $\gr_{F_\mathbb{Z}}(U(L))$. Therefore
the isomorphism given in \eqref{eq:isomorphismofgraded2} is an isomorphism of
$k$-algebras with involution, that is, if $\Phi$ is the isomorphism of \eqref{eq:isomorphismofgraded2}
then $\Phi(f^*)=\Phi(f)^*$. 

Observe that 
$\gr_{F_\mathbb{Z}}(L)$ is a residually nilpotent Lie $k$-algebra.  Define 
$N=\bigoplus_{n\geq 4}L_n$. Then  
$\gr_{F_\mathbb{Z}}(L)/N$ is isomorphic to the Heisenberg Lie $k$-algebra $H$. Moreover $N$
is invariant under the involution $*$, and the induced involution in $\gr_{F_\mathbb{Z}}(L)/N$
is the one in Lemma~\ref{lem:equivalentinvolutionHeisenbergalgebra}(1).

The valuation  $\upsilon\colon U(L)\rightarrow
\mathbb{Z}\cup\{\infty\}$ can be
extended to a valuation $\upsilon\colon \mathfrak{D}(L)\rightarrow
\mathbb{Z}\cup\{\infty\}$ \cite[Proposition~9.1.1]{Cohnskew}.

Consider $u,v$ and $w=[v,u]$. Note that $\upsilon(u)=-1,\upsilon(v)=-2$
and $\upsilon(w)=-3$.  Denote by
$\bar{u}$, $\bar{v}$, $\bar{w}$ the class of $u\in U(L)_{-1}$, $v\in U(L)_{-2}$,
$w\in U(L)_{-3}$ and also the
class of $u$ in $L_{-1}$, $v\in L_{-2}$ and $w\in L_{-3}$, respectively.   By
Lemma~\ref{lem:gradedOre}(4), $U(\gr_{F_\mathbb{Z}}(L))$ is an Ore domain.
Let $\mathfrak{D}(\gr_{F_\mathbb{Z}}(L))$ be its Ore ring of fractions.
Observe that $[\bar{v},\bar{u}]=\bar{w}$ as elements of
$\gr_{F_\mathbb{Z}}(L)$.
Define
$\overline{V}=\frac{1}{2}(\bar{u}\bar{v}+\bar{v}\bar{u})$, and consider the following
elements of $\mathfrak{D}(\gr_{F_\mathbb{Z}}(L))$:
$$\overline{S}_1=\bar{w}^{-1}\Big((\overline{V}-\frac{1}{3}\bar{w})
(\overline{V}+\frac{1}{3}\bar{w})^{-1}+ (\overline{V}-\frac{1}{3}\bar{w})^{-1}(\overline{V}+\frac{1}{3}\bar{w})\Big) \bar{w}^{-1},$$
 $$\overline{T}_2=(\bar{w}^2+\bar{v}^3)^{-1}(\bar{w}^2-\bar{v}^3)\overline{S}_1(\bar{w}^2+\bar{v}^3)(\bar{w}^2-\bar{v}^3)^{-1}.$$

By Corollary~\ref{coro:freesymmetricresiduallynilpotent}(3)(ii)(b), 
the $k$-subalgebra of $\mathfrak{D}(\gr_{F_\mathbb{Z}}(L))$ 
generated by $\{\overline{S}_1\overline{S}_1^*, \overline{T}_2\overline{T}_2^*\}$ is the free 
 $k$-algebra on $\{\overline{S}_1\overline{S}_1^*, \overline{T}_2\overline{T}_2^*\}$.
Let $\mathcal{H}$ be the set of homogeneous elements of $\gr_{F_\mathbb{Z}}(U(L))$. From
\eqref{eq:isomorphismofgraded2} and Lemma~\ref{lem:gradedOre}, we
obtain the following commutative diagram
$$\xymatrix{\gr_{F_\mathbb{Z}}(U(L))\cong U(\gr_{F_\mathbb{Z}}(L))\ar@{^{(}->}[r]\ar@{^{(}->}[d]
 &
\mathfrak{D}(\gr_{F_\mathbb{Z}}(L))\\
\mathcal{H}^{-1}\gr_{F_\mathbb{Z}}(U(L))\cong\gr_{F_\mathbb{Z}}(\mathfrak{D}(L))
\ar@{^{(}->}[ur] & },$$ where the diagonal arrow is obtained from the
universal property of the Ore localization. Note that
$\overline{V},\, \overline{V}-\frac{1}{3}\bar{w},\,
\overline{V}+\frac{1}{3}\bar{w}$ are homogeneous elements of degree $-3$,
and the elements  $\bar{w}^2+\bar{v}^3,\,
\bar{w}^2-\bar{v}^3$ are homogeneous elements of degree $-3$ in
$\gr_{F_\mathbb{Z}}(U(L))$. Thus $\overline{S}_1$,
$\overline{S}_1^*$, $\overline{T}_2$,
$\overline{T}_2^*$ are in fact homogeneous elements
of degree $-6$ in $\gr_{F_\mathbb{Z}}(\mathfrak{D}(L))$.

Now observe that $S_1$, $S_1^*$, $T_2$ and $T_2^*$ are elements of $\mathfrak{D}(L)$ such
that $\upsilon(S_1)=\upsilon(S_1^*)=\upsilon(T_2)=\upsilon(T_2^*)=6$,
hence $\upsilon(S_1S_1^*)=12$, $\upsilon(T_2T_2^*)=12$  and
$\overline{S}_1\overline{S}_1^*=S_1S_1^*+\mathfrak{D}(L)_{>12},\,
\overline{T}_2\overline{T}_2^*=T_2T_2^*+\mathfrak{D}(L)_{>12}$ in
$\gr_{F_\mathbb{Z}}(\mathfrak{D}(L))$.

Now, by Theorem~\ref{coro:divisionrings}, the result follows.
\end{proof}

In case that $[x,x^*]=0$ for all $x\in L$ we are able to prove the following.

\begin{theo}\label{theo:symmetricOre2}
Let $k$ be a field of characteristic zero and $L$ be a Lie $k$-algebra such that
$U(L)$ is an Ore domain. Let $*\colon L\rightarrow L$ be a $k$-involution. 
Suppose that $[x,x^*]=0$ for all $x\in L$, but there exist elements
$x,y\in L$ such that $[y,x]\neq 0$ and the $k$-subspace of $L$ spanned by
$\{x,x^*,y,y^*\}$ is not the Lie $k$-subalgebra generated by $\{x,x^*,y,y^*\}$.
Then there exist symmetric elements $A,B\in\mathfrak{D}(L)$ such that the $k$-subalgebra
generated by $\{A,A^{-1},B,B^{-1}\}$ is the free group $k$-algebra on $\{A,B\}$. 
\end{theo}

\begin{proof}
Let $L_1$ be the Lie $k$-subalgebra of $L$ generated by $\{x,x^*,y,y^*\}$.
Since $U(L)$ is an Ore domain, $U(L_1)$ is also an Ore domain.
Moreover, $\mathfrak{D}(L_1)\subseteq \mathfrak{D}(L)$. Thus, we may
suppose that $L$ is generated by $\{x,x^*,y,y^*\}$. Let $V$ be the $k$-subspace
spanned by $\{x,x^*,y,y^*\}$. 
Consider the filtration $F_{\mathbb{Z}}L=\{F_nL\}_{n\in\mathbb{Z}}$ of $L$
defined by 
$F_rL=0$ for all $r\geq 0$,  $F_{-1}L=V$, $F_{-2}L=[V,V]+F_{-1}L$, 
and for $n\leq -2$,
$$F_{n-1}L=\sum_{n_1+n_2+\dotsb+n_r\geq (n-1)}[F_{n_1}L,[F_{n_2}L,\dotsc] \dotsb].$$
Note that $F_{n}L$ is invariant under $*$ for all $n\in\mathbb{Z}$. Thus the involution
on $L$ induces an involution on $\gr_{F_\mathbb{Z}}(L)$. Define now 
$N=\bigoplus_{n\leq -3}L_n$. Then $N$ is an ideal of $\gr_{F_\mathbb{Z}}(L)$  such that
$\gr_{F_\mathbb{Z}}(L)/N$ is a nonabelian nilpotent Lie $k$-algebra because
$[V,V]$ is not contained in $V$ by assumption. Moreover $N$
is invariant under $*$ and thus the involution on $\gr_{F_\mathbb{Z}}(L)$ induces an involution
on $\gr_{F_\mathbb{Z}}(L)/N$ again denoted by $*$. By Proposition~\ref{prop:involutionnilpotent},
there exist $u,v\in \gr_{F_\mathbb{Z}}(L)/N$ such that they generate
a $*$-invariant Heisenberg Lie $k$-subalgebra of $\gr_{F_\mathbb{Z}}(L)/N$
and the restriction to it is one of the involutions in Lemma~\ref{lem:equivalentinvolutionHeisenbergalgebra}.
Note that $F_{-1}L=L_{-1}$. Also 
$\gr_{F_\mathbb{Z}}(L)/N\cong L_{-1}\oplus L_{-2}$ as $k$-vector spaces, 
and the induced product $[L_{-1},L_{-2}]=0$.
Thus we can choose $u,v\in L_{-1}=F_{-1}L$.

Suppose that the involution on $\gr_{F_\mathbb{Z}}(L)/N$ is like the one in 
Lemma~\ref{lem:equivalentinvolutionHeisenbergalgebra}(1), i.e. $u^*=u$, $v^*=-v$ 
and $w^*=w$, where $w=[v,u]$. Then take $u_1=u+v,\, v_1=u-v\in L_{-1}$.
Note that $u_1^*=v_1$  and $[u_1,v_1]=2[v,u]\neq 0$, a contradiction
with our assumption that $[x,x^*]=0$ for all $x\in L$. Hence the involution on
the Heisenberg subalgebra of $\gr_{F_\mathbb{Z}}(L)/N$ generated by $u,v$
is of type either Lemma~\ref{lem:equivalentinvolutionHeisenbergalgebra}(2)
or Lemma~\ref{lem:equivalentinvolutionHeisenbergalgebra}(3).

Let $L_2$ be the Lie $k$-subalgebra of $L$ generated by $\{u,v\}$.
Since $U(L)$ is an Ore domain, $U(L_2)$ is also an Ore domain.
Moreover, $\mathfrak{D}(L_2)\subseteq \mathfrak{D}(L)$. Thus, we may
suppose that $L$ is generated by $\{u,v\}$. Let $V$ be the $k$-subspace
spanned by $\{u,v\}$. 
Consider the filtration $F_{\mathbb{Z}}L=\{F_nL\}_{n\in\mathbb{Z}}$ of $L$
defined by 
$F_rL=0$ for all $r\geq 0$,  $F_{-1}L=V$, $F_{-2}L=[V,V]+F_{-1}L$, 
and for $n\leq -2$,
$$F_{n-1}L=\sum_{n_1+n_2+\dotsb+n_r\geq (n-1)}[F_{n_1}L,[F_{n_2}L,\dotsc] \dotsb].$$
Note that $F_{n}L$ is invariant under $*$ for all $n\in\mathbb{Z}$. Thus the involution
on $L$ induces an involution on $\gr_{F_\mathbb{Z}}(L)$. Define now 
$N=\bigoplus_{n\leq -3}L_n$. Then $N$ is an ideal of $\gr_{F_\mathbb{Z}}(L)$  such that
$\gr_{F_\mathbb{Z}}(L)/N$ is the Heisenberg Lie $k$-algebra and the involution induced on it
is of type either Lemma~\ref{lem:equivalentinvolutionHeisenbergalgebra}(2)
or Lemma~\ref{lem:equivalentinvolutionHeisenbergalgebra}(3).

Observe that, for each $n\in\mathbb{Z}$, there exists   $\mathcal{B}_{n}\subseteq L$
whose classes give a basis of $L_{n}=F_nL/F_{n+1}L$ such that
$\bigcup\limits_{n\in\integers} \mathcal{B}_{n}$ is a basis of $L$.
This filtration on $L$ induces a filtration 
$F_\mathbb{Z}U(L)=\{F_nU(L)\}_{n\in\mathbb{Z}}$  on $U(L)$ as shown
in Section~\ref{sec:filtrationuniversal}. Moreover, by
Lemma~\ref{lem:filtrationuniversalenveloping}, there exists an isomorphism of $\mathbb{Z}$-graded $k$-algebras
\begin{equation}\label{eq:isomorphismofgraded3}
U(\gr_{F_\mathbb{Z}}(L)) \cong  \gr_{F_\mathbb{Z}}(U(L)).
\end{equation}
 which induces a valuation $\upsilon\colon U(L)\rightarrow
\mathbb{Z}\cup\{\infty\}$ as in Section~\ref{sec:generalfiltrations}.
In what follows,  the two
objects in \eqref{eq:isomorphismofgraded3} will be identified
via the isomorphism given in either \cite[Proposition~1]{Vergne} or 
\cite[Lemma~2.1.2]{Boiscorpsenveloppants}. This isomorphism sends
the class of an element of $\mathcal{B}_n$ in $L_n$ to
its class in $U(L)_n$.

The valuation $\upsilon\colon U(L)\rightarrow
\mathbb{Z}\cup\{\infty\}$ can be
extended to a valuation $\upsilon\colon \mathfrak{D}(L)\rightarrow
\mathbb{Z}\cup\{\infty\}$ \cite[Proposition~9.1.1]{Cohnskew}. 

Consider $u,v$ and $w=[v,u]$. Note that $\upsilon(u)=\upsilon(v)=-1$
and $\upsilon(w)=-2$ because $L$ is not two-dimensional.  Denote by
$\bar{u}$, $\bar{v}$ the class of $u,v\in U(L)_{-1}$ and also the
class of $u$ and $v$ in $L_{-1}$. Denote by $\bar{w}$ the class of
$w$ in $U(L)_{-2}$ and in $L_{-2}$.  By
Lemma~\ref{lem:gradedOre}(4), $U(\gr_{F_\mathbb{Z}}(L))$ is an Ore domain.
Let $\mathfrak{D}(\gr_{F_\mathbb{Z}}(L))$ be its Ore ring of fractions.

Observe that $[\bar{v},\bar{u}]=\bar{w}$ as elements of
$\gr_{F_\mathbb{Z}}(L)$.
Define
$\overline{V}=\frac{1}{2}(\bar{u}\bar{v}+\bar{v}\bar{u})$, and consider the following
elements of $\mathfrak{D}(\gr_{F_\mathbb{Z}}(L))$:
$$\overline{S}_1=\bar{w}^{-1}\Big((\overline{V}-\frac{1}{3}\bar{w})
(\overline{V}+\frac{1}{3}\bar{w})^{-1}+ (\overline{V}-\frac{1}{3}\bar{w})^{-1}(\overline{V}+\frac{1}{3}\bar{w})\Big) \bar{w}^{-1},$$
 $$\overline{T}_1=(\bar{w}+\bar{v}^2)^{-1}(\bar{w}-\bar{v}^2)\overline{S}_1(\bar{w}+\bar{v}^2)(\bar{w}-\bar{v}^2)^{-1}.$$

By Corollary~\ref{coro:freesymmetricresiduallynilpotent}(3)(i)(b), 
the $k$-subalgebra of $\mathfrak{D}(\gr_{F_\mathbb{Z}}(L))$ 
generated by $\{\overline{S}_1\overline{S}_1^*, \overline{T}_1\overline{T}_1^*\}$ is the free 
 $k$-algebra on $\{\overline{S}_1\overline{S}_1^*, \overline{T}_2\overline{T}_2^*\}$.
Let $\mathcal{H}$ be the set of homogeneous elements of $\gr_{F_\mathbb{Z}}(U(L))$. From
\eqref{eq:isomorphismofgraded3} and Lemma~\ref{lem:gradedOre}, we
obtain the following commutative diagram
$$\xymatrix{\gr_{F_\mathbb{Z}}(U(L))\cong U(\gr_{F_\mathbb{Z}}(L))\ar@{^{(}->}[r]\ar@{^{(}->}[d]
 &
\mathfrak{D}(\gr_{F_\mathbb{Z}}(L))\\
\mathcal{H}^{-1}\gr_{F_\mathbb{Z}}(U(L))\cong\gr_{F_\mathbb{Z}}(\mathfrak{D}(L))
\ar@{^{(}->}[ur] & },$$ where the diagonal arrow is obtained from the
universal property of the Ore localization. Note that
$\overline{V},\, \overline{V}-\frac{1}{3}\bar{w},\,
\overline{V}+\frac{1}{3}\bar{w}$ are homogeneous elements of degree $-3$,
and the elements  $\bar{w}+\bar{v}^2,\,
\bar{w}-\bar{v}^2$ are homogeneous elements of degree $-2$ in
$\gr_{F_\mathbb{Z}}(U(L))$. Thus $\overline{S}_1$,
$\overline{S}_1^*$, $\overline{T}_1$,
$\overline{T}_1^*$ are in fact homogeneous elements
of degree $-4$ in $\gr_{F_\mathbb{Z}}(\mathfrak{D}(L))$.

Now observe that $S_1$, $S_1^*$, $T_2$ and $T_2^*$ are elements of $\mathfrak{D}(L)$ such
that $\upsilon(S_1)=\upsilon(S_1^*)=\upsilon(T_2)=\upsilon(T_2^*)=4$,
hence $\upsilon(S_1S_1^*)=8$, $\upsilon(T_2T_2^*)=8$  and
$\overline{S}_1\overline{S}_1^*=S_1S_1^*+\mathfrak{D}(L)_{>8},\,
\overline{T}_2\overline{T}_2^*=T_2T_2^*+\mathfrak{D}(L)_{>8}$ in
$\gr_{F_\mathbb{Z}}(\mathfrak{D}(L))$.

Defining $A=1+S_1S_1^*$ and $B=1+T_1T_1^*$, the result follows from Theorem~\ref{coro:divisionrings}.
\end{proof}

As a corollary, we obtain a generalization of \cite[Theorem~5.2]{FerreiraGoncalvesSanchez2},
where the existence of a free $k$-algebra was proved. We recall that
the principal involution on a Lie $k$-algebra $L$ is defined by $L\rightarrow L$, $f\mapsto -f$. 

\begin{coro}\label{coro:symmeticprincipalinvolution}
Let $k$ be a field of characteristic zero and $L$ be a Lie $k$-algebra such that its
universal enveloping algebra $U(L)$ is an Ore domain. Let
$\mathfrak{D}(L)$ be its Ore ring of fractions. Let $u,v\in L$ be such that
the Lie subalgebra generated by them is of dimension at least three.  
Then there exist symmetric elements $A,B\in\mathfrak{D}(L)$ with respect
to the principal involution such that the $k$-subalgebra
generated by $\{A,A^{-1},B,B^{-1}\}$ is the free group $k$-algebra on $\{A,B\}$. \qed
\end{coro}

\section{Free group algebras in the Malcev-Neumann division ring of fractions
of a residually torsion-free nilpotent group}\label{sec:Heisenberggroup}

In this section, for a group $G$ and elements $x,y\in G$, $(y,x)$ denotes the
commutator $(y,x)=y^{-1}x^{-1}yx$. Also, if $A,B$ are subgroups of $G$, $(B,A)$
denotes the subgroup of $G$ generated by the commutators $(y,x)$ with $y\in B$, $x\in A$.

\medskip

Let $R$ be a ring and  $(G,<)$ be an ordered group. Suppose that
$R[G]$ is the group ring of $G$ over $R$.
We define a new ring,
denoted $R((G;<))$ and called \emph{Malcev-Neumann
series ring}, in which $R[G]$ embeds. As a set,
\[
R((G;<))=\Bigl\{f=\sum_{x\in G}a_x x :  a_x\in R,\ \supp(f) 
\text{ is well-ordered}\Bigr\},
\] where $\supp(f)=\{x\in G\mid a_x\neq 0\}$.
Addition and multiplication are defined extending the ones in
$R[G]$. That is, given $f=\sum\limits_{x\in G}a_x x
$ and $g=\sum\limits_{x\in G}b_x x$, elements of
$R((G;<))$, one has
\[
f+g=\sum_{x\in G}(a_x+b_x) x \quad\text{and}\quad
fg=\sum_{x\in G}\Bigl(\sum_{yz=x}a_yb_z \Bigr)x.
\]

It was shown, independently, in \cite{Malcev}
and \cite{Neumann} that if 
$R$ is a division ring, then
$R((G;<))$ is a division ring. 

If $k$ is a  field, the division subring of $k((G;<))$
generated by the group ring $k[G]$ will be
called the \emph{Malcev-Neumann division ring of fractions} of
$k[G]$ and it will be denoted by $k(G)$. It
is important to observe the following. For a subgroup $H$ of $G$,
$k((H;<))$ and $k(H)$ can be regarded
as division subrings of  $k((G;<))$ and
$k(G)$, respectively, in the natural way.  We remark that
$k(G)$ does not depend on the order $<$ of $G$,
see \cite{Hughes}. When the group ring $k[G]$ is 
an Ore domain, then $k(G)$ is the Ore ring of fractions of $k[G]$. 

\medskip

An \emph{involution} \label{involutiongroup}
on a group $G$ is a map $*\colon G\rightarrow G$, $x\mapsto x^*$,
that satisfies $$(xy)^*=y^*x^*\ \textrm{ and } \ (x^*)^*=x \textrm{ for all } x,y\in G.$$
In other words, $*$ is an anti-automorphism of order two. 

Suppose that $G$ is a group endowed with an involution $*\colon G\rightarrow G$, $x\mapsto x^*$,
$k$ is a field and $k[G]$ is the group $k$-algebra. The map
$*\colon k[G]\rightarrow k[G]$ defined by $\Big(\sum\limits_{x\in G} a_xx\Big)^*=\sum\limits_{x\in G}
a_xx^*$ is a $k$-involution on $k[G]$.

If $(G,<)$ is an ordered group, we remark that the $k$-involution on the group algebra $k[G]$ 
induced the involution
$\ast$ on $G$ extends uniquely to a $k$-involution on the Malcev-Neumann
division ring of fractions $k(G)$ of $k[G]$,
see \cite[Theorem~2.9]{FerreiraGoncalvesSanchez1}.

\medskip

Let $G$ be a group. If $H$ is a subgroup of $G$, we denote by $\sqrt{H}$ the subset of $G$ defined by
\[\sqrt{H}=\{x\in G : x^n\in H \text{, for some } n\geq 1\}.\]
We shall denote the $n$-th
term of the lower central series of $G$ by $\gamma_n(G)$. That is,
we set $\gamma_1(G)=G$ and, for $n\ge 1$, define 
$\gamma_{n+1}(G)=(G,\gamma_n(G))$.

A group $G$ is \emph{residually torsion-free nilpotent} if for
each $g\in G$, there exists a normal subgroup $N_g$ of $G$ such that
$g\notin N_g$ and $G/N_g$ is torsion-free nilpotent. Equivalently,
$\bigcap_{n\geq 1}\sqrt{\gamma_n(G)}=\{1\}$. 
It is well known that any residually torsion-free nilpotent group  is orderable, see
for example \cite[Theorem~IV.6]{Fuchs}.

Let $G$ be a residually torsion-free nilpotent group, $k$ be a field of characteristic zero
and consider the group algebra $k[G]$. Consider an involution on $G$
and its extension to the Malcev-Neumann
division ring of fractions $k(G)$ of $k[G]$. 
The aim of this section is to prove that
 there exist symmetric elements in $k(G)$ that generate a noncommutative
free group $k$-algebra. For that we will need the following discussion.

\bigskip

Let $G$ be a torsion-free nilpotent group. An \emph{$\mathcal{N}$-series} of $G$ is a sequence
$\{H_i\}_{i\geq 1}$, 
$$G=H_1\supseteq H_2\supseteq\dotsb\supseteq H_n\supseteq\dotsb$$
of normal subgroups of $G$ that satisfies the following three conditions 
$$(H_i,H_j)\subseteq H_{i+j},\quad \bigcap_{i\geq 1}H_i=\{1\},\quad G/H_i \textrm{ is torsion free for all } i\geq 1.$$
The $\mathcal{N}$-series induces a \emph{weight function} $w\colon G\rightarrow \mathbb{N}\cup\{\infty\}$ defined
by $w(g)=i$ if $g\in H_i\setminus H_{i+1}$ and $w(1)=\infty.$

Let $k$ be a field of characteristic zero, $G$ torsion-free nilpotent group
with an $\mathcal{N}$-series $\{H_i\}_{i\geq 1}$ and consider the group ring $k[G]$. The
$\mathcal{N}$-series defines the \emph{cannonical filtration}, $F_\mathbb{Z}k[G]=\{F_nk[G]\}_{n\in\mathbb{Z}}$, 
induced by $\{H_i\}_{i\geq 1}$ which is defined by
$F_nk[G]=k[G]$, for all $n\leq 0$, and for $n\geq 1$, $F_nk[G]$ is the $k$-vector space
spanned by the set $$\Big\{(h_1-1)(h_2-1)\cdots(h_s-1)\colon s\geq 1,\ \sum_{j=1}^s w(h_j)\geq n\Big\}$$ 
Note that $F_1k[G]$ is the augmentation ideal of $k[G]$ and that $F_nk[G]\cdot F_mk[G]\subseteq F_{n+m}k[G]$.

For each $i\geq 1$, $H_i/H_{i+1}$ is an abelian group. Denote the operation
additively. More precisely, if $x_i,x_i'\in H_i$, $\widetilde{x_i}$ denotes the class $x_iH_{i+1}$.
Then $\widetilde{x_i}+\widetilde{x_i'}=\widetilde{x_ix_i'}$ in $H_i/H_{i+1}$.
The abelian group $L(G)=\bigoplus\limits_{i\geq 1}H_i/H_{i+1}$ can be endowed
with a $\mathbb{Z}$-graded
Lie $\mathbb{Z}$-algebra structure with the following product on homogeneous elements
$[\widetilde{x_i},\widetilde{x_j}]=\widetilde{(x_i,x_j)}\in H_{i+j}/H_{i+j+1}$,
for $x_i\in H_i$, $x_j\in H_j$, and then extending by bilinearity. Then
$k\otimes_\mathbb{Z}L(G)$ is a Lie $k$-algebra with universal enveloping algebra
$U(k\otimes_\mathbb{Z}L(G))$. 

In \cite[Theorem~2.3]{LichtmanmatrixringsI},
Lichtman proved  
a more general version of \cite{Quillengraded},
in a similar way as  Quillen's result is shown in \cite[Chapter~VIII]{Passibookaugmentation}. 
Lichtman's result implies that that
there exists an isomorphism of
$\mathbb{Z}$-graded $k$-algebras 
\begin{eqnarray}\label{eq:isomorphismQuillen}
\Theta\colon U(k\otimes_\mathbb{Z}L(G)) &\rightarrow &\gr_{F_\mathbb{Z}}(k[G]) \\
 \widetilde{x_i} & \mapsto &(x_i-1)+F_{i+1}k[G]. \nonumber
\end{eqnarray}
\bigskip

Let $\mathbb{H}=\langle a,b\mid (b,(b,a))=(a,(b,a))=1 \rangle$ be the Heisenberg group. 
Define $c=(b,a)$. 
Consider the following \emph{main involutions} of $\mathbb{H}$ which are defined on the generators and
extended accordingly. 
\begin{enumerate}[(1)]
	\item $a^*=a$, $b^*=b^{-1}$ and $c^*=c$.
	\item $a^*=a$, $b^*=b$ and $c^*=c^{-1}$.
	\item $a^*=a^{-1}$, $b^*=b^{-1}$ and $c^*=c^{-1}$.
\end{enumerate}

\begin{prop}\label{prop:symmetricHeisenberggroup}
Let $k$ be a field of characteristic zero. 
Let $\mathbb{H}=\langle a,b\mid (b,(b,a))=(a,(b,a))=1 \rangle$ be the Heisenberg group and $c=(b,a)$. 
Consider the
group $k$-algebra $k[\mathbb{H}]$ and its Ore ring of fractions $k(\mathbb{H})$. 
Consider the elements of $k(H)$
$$\mathcal{V}=\frac{1}{2}((a-1)(b-1)+(b-1)(a-1)),$$
$$\mathcal{S}_2=(c-1)\Big((\mathcal{V}-\frac{1}{3}(c-1))(\mathcal{V}+\frac{1}{3}(c-1))^{-1}+
(\mathcal{V}-\frac{1}{3}(c-1))^{-1}(\mathcal{V}+\frac{1}{3}(c-1))\Big)(c-1),$$
$$\mathcal{T}_3=((c-1)+(b-1)^2)^{-1}((c-1)-(b-1)^2) \mathcal{S}_2 ((c-1)+(b-1)^2)((c-1)-(b-1)^2)^{-1},$$
$$\mathcal{T}_4= ((c-1)^2+(b-1)^3)^{-1}((c-1)^2-(b-1)^3)\mathcal{S}_2((c-1)^2+(b-1)^3)((c-1)^2-(b-1)^3)^{-1}.$$
The following statements hold true.
\begin{enumerate}[\rm(1)]
	\item Suppose that $*\colon \mathbb{H}\rightarrow \mathbb{H}$ is one of the main involutions (2) or (3) above.
	Then the $k$-subalgebra of $k(\mathbb{H})$ generated by
	$$\{1+\mathcal{S}_2\mathcal{S}_2^*, 
	(1+\mathcal{S}_2\mathcal{S}_2^*)^{-1}, (1+\mathcal{T}_3\mathcal{T}_3^*), (1+\mathcal{T}_3\mathcal{T}_3^*)^{-1} \}$$
	is the free group $k$-algebra on the set $\{1+\mathcal{S}_2\mathcal{S}_2^*, 
	 (1+\mathcal{T}_3\mathcal{T}_3^*)\}$.

	\item Suppose that $*\colon \mathbb{H}\rightarrow \mathbb{H}$ is the main involution (1) above.
	Then the $k$-subalgebra of $k(\mathbb{H})$ generated by
	$$\{1+\mathcal{S}_2\mathcal{S}_2^*, 
	(1+\mathcal{S}_2\mathcal{S}_2^*)^{-1}, (1+\mathcal{T}_4\mathcal{T}_4^*), (1+\mathcal{T}_4\mathcal{T}_4^*)^{-1} \}$$
	is the free group $k$-algebra on the set $\{1+\mathcal{S}_2\mathcal{S}_2^*, 
	 (1+\mathcal{T}_4\mathcal{T}_4^*)\}$.
\end{enumerate}
\end{prop}

\begin{proof}
(1) Consider the following $\mathcal{N}$-series of $\mathbb{H}$. 
$$H_1=\mathbb{H}\supseteq H_2=(c)\supseteq H_3=\{1\}.$$
If we set $x=aH_2,y=bH_2\in H_1/H_2$ and $z=cH_3\in H_2/H_3$, 
then the $\mathbb{Z}$-graded Lie $\mathbb{Z}$-algebra $L(H)$ has as
$\mathbb{Z}$-basis the elements $x,y,z$ with products $[y,x]=z$, $[y,z]=[x,z]=0$.
Hence the $\mathbb{Z}$ graded Lie $k$-algebra $k\otimes_\mathbb{Z} L(\mathbb{H})$ is the
Heisenberg Lie $k$-algebra $H$ with the
$\mathbb{Z}$-grading given in Example~\ref{ex:gradedLie}(c). 
The isomorphism \eqref{eq:isomorphismQuillen} implies that the cannonical filtration
induced by the $\mathcal{N}$-series is in fact a valuation, because the graded ring is
a domain. Since $k[\mathbb{H}]$ is an Ore domain, the valuation can be extended to
a valuation $\upsilon\colon k(\mathbb{H})\rightarrow \mathbb{Z}\cup\{\infty\}$.
If we let $\mathcal{H}$ be the homogeneous elements of $\gr_{F_\mathbb{Z}}(k[\mathbb{H}])$,
Lemma~\ref{lem:Reesgradedring}(3) implies that there exists an isomorphism
of $\mathbb{Z}$-graded $k$-algebras
\begin{equation}\label{eq:isoQuillenHeisenberg}
\gr_\upsilon(k(\mathbb{H}))\cong 
\mathcal{H}^{-1}\gr_{F_\mathbb{Z}}(k[\mathbb{H}])\cong \mathcal{H}^{-1}U(k\otimes_\mathbb{Z}L(\mathbb{H})).
\end{equation} 
Observe that $\mathcal{H}^{-1}U(k\otimes_\mathbb{Z}L(\mathbb{H}))\hookrightarrow 
\mathfrak{D}(k\otimes_\mathbb{Z}L(\mathbb{H}))$. Now note that 
$$\mathcal{V}, \mathcal{V}\pm(c-1), (c-1), (b-1)^2, (c-1)\pm(b-1)^2 \in F_2k[\mathbb{H}]\setminus F_3k[\mathbb{H}].$$
Hence the classes of these elements in $\gr_{F_\mathbb{Z}}(k[\mathbb{H}])$ are homogeneous of degree two.
It implies that the class of $\mathcal{S}_2$ and $\mathcal{T}_3$ in $\gr_{\upsilon}(k(\mathbb{H}))$
are homogeneous of degree four. Moreover, their image under the isomorphism \eqref{eq:isoQuillenHeisenberg}
are the elements $S_2,T_3\in \mathfrak{D}(k\otimes_\mathbb{Z}L(\mathbb{H}))$ given in 
Theorem~\ref{theo:freegroupHeisenberg}(4).

Since each $H_i$ is invariant under the involution $*$, it induces a $k$-involution
in the Lie $k$-algebra $k\otimes_\mathbb{Z}L(\mathbb{H})$. Hence the isomorphism \eqref{eq:isoQuillenHeisenberg}
is an isomorphism of $*$-algebras, i.e. $\Phi(f^*)=\Phi(f)^*$.
Note that the induced involution on $k\otimes_\mathbb{Z}L(\mathbb{H})$ is one of the
involutions in Lemma~\ref{lem:equivalentinvolutionHeisenbergalgebra}(2) or (3). By 
Theorem~\ref{theo:freegroupHeisenberg}(4)(a), the elements $S_2$, $T_2$ are symmetric with respect to 
the induced involution on $\mathfrak{D}(k\otimes_\mathbb{Z}L(\mathbb{H}))$. 
Hence, the image of the classes of $\mathcal{S}_2^*$ and $\mathcal{T}_3^*$ are also $S_2$
and $T_3$, respectively. The classes
of the elements $\mathcal{S}_2\mathcal{S}_2^*, \mathcal{T}_3\mathcal{T}_3^*$ 
in $\gr_{\upsilon}(k(\mathbb{H}))$
are homogeneous of degree $8$. Moreover,  they generate a free algebra in 
$\gr_{\upsilon}(k(\mathbb{H}))$, because $S_2^2$ and $T_3^2$ generate a free algebra
in $\mathfrak{D}(k\otimes_k L(\mathbb{H}))$ by Theorem~\ref{theo:freegroupHeisenberg}(4)(b).
Now the result follows by Theorem~\ref{coro:divisionrings}.

(2) It follows in the same way as (1). Now one has to consider the $\mathcal{N}$-series
$$H_1=G\supseteq H_2=\langle b,c\rangle\supseteq H_3=\langle c\rangle\supseteq H_4=\{1\}.$$
Then again $k\otimes_\mathbb{Z} L(H)$ is the Heisenberg Lie $k$-algebra, but with the gradation
given in Example~\ref{ex:gradedLie}(4). Then
the isomorphism in \eqref{eq:isoQuillenHeisenberg} (with a
different gradation) sends $\mathcal{S}_2$ and $\mathcal{T}_4$
to the elements $S_2$ and $T_4$ in Theorem~\ref{theo:freegroupHeisenberg}(5).
\end{proof}

The next result is 
\cite[Proposition~2.4]{FerreiraGoncalvesSanchezFreegroupssymmetric}

\begin{prop}\label{prop:InvHeis}
Let $G$ be a nonabelian torsion-free nilpotent group with involution $\ast$. Then $G$ 
contains a $\ast$-invariant Heisenberg subgroup $\mathbb{H}$ such that
the induced involution  is one of the main involutions of $\mathbb{H}$. More precisely, 
there exist $x,y\in G$ such that $(x,y)\ne 1$, $(x,(x,y))=(y,(x,y))=1$, 
$x^{\ast}=x^{\pm 1}, y^{\ast}=y^{\pm 1}$. \qed
\end{prop}

Recall that given a group $G$ and a field $k$ such that $k[G]$ is an Ore domain,
then $k[N]$ is an Ore domain for any subgroup $N$ of $G$. Hence, if $G$ is a torsion-free nilpotent group
and $\mathbb{H}$ is a subgroup of $G$, then $k(\mathbb{H})$ is embedded in $k(G)$.
This fact, together with Propositions~\ref{prop:InvHeis} and \ref{prop:symmetricHeisenberggroup},
imply the following result.

\begin{theo}\label{theo:symmetricnilpotentgroup}
Let $G$ be a nonabelian torsion-free nilpotent group with an involution $\ast\colon G\rightarrow
G$ and $k$ be a field
of characteristic zero. Consider the group ring $k[G]$ and its Ore ring of fractions $k(G)$. 
Then there exist nonzero symmetric elements $A,B\in k(G)$ such that the $k$-subalgebra 
generated by $\{A,A^{-1},B,B^{-1}\}$ is the free group $k$-algebra on the set $\{A,B\}$. \qed
\end{theo}

\begin{theo}
 Let $G$ be a residually torsion-free nilpotent group with an involution $\ast\colon G\rightarrow
G$ and  $k$ be a field
of characteristic zero. Consider the group ring $k[G]$ and its Malcev-Neumann
division ring of fractions $k(G)$. Then there exist nonzero symmetric
elements $A,B\in k(G)$ such that the $k$-subalgebra 
generated by $\{A,A^{-1},B,B^{-1}\}$ is the free group $k$-algebra on the set $\{A,B\}$. 
\end{theo}

\begin{proof}
As noted in \cite[Section~3]{FerreiraGoncalvesSanchezFreegroupssymmetric}, the argument used there
can also be used to prove the existence of  free group algebras generated symmetric elements in $k(G)$ using 
the existence of free group algebras generated by symmetric elements in Ore ring of fractions
$k(L)$ where $L$ is a torsion-free nilpotent group. This fact has already been proved in
Theorem~\ref{theo:symmetricnilpotentgroup}.
\end{proof}

\bibliographystyle{amsplain}
\bibliography{grupitosbuenos}

\end{document}